\documentclass{article}

\usepackage{arxiv}

\usepackage[utf8]{inputenc} 
\usepackage[T1]{fontenc}    
\usepackage{hyperref}       
\usepackage{url}            
\usepackage{booktabs}       
\usepackage{amsmath}
\usepackage{amsfonts}       
\usepackage{nicefrac}       
\usepackage{microtype}      
\usepackage{cleveref}       
\usepackage{lipsum}         
\usepackage{graphicx}
\usepackage{doi}
\usepackage{amssymb,amscd,pb-diagram, amsthm}
\usepackage[all,arc]{xy}
\usepackage{tabularx}
\usepackage{xcolor}

\newtheorem{theorem}{Theorem}[section]

\newtheorem{lemma}{Lemma}[section]
\newtheorem{corollary}{Corollary}[section]

\newtheorem{proposition}{Proposition}[section]
\newtheorem{definition}{Definition}[section]
\newtheorem{remark}{Remark}[section]
\newtheorem{example}{Example}[section]

\title{Hartogs and open embeddings, proper maps, compactifications, cohomologies}

\date{}

\newif\ifuniqueAffiliation
\uniqueAffiliationtrue

\ifuniqueAffiliation 
\author{{\hspace{1mm}S. Feklistov} \\
}

\renewcommand{\headeright}{}
\renewcommand{\shorttitle}{Hartogs and open embeddings, proper maps, compactifications, cohomologies}

\hypersetup{
pdftitle={Hartogs, open embeddings, proper maps, compactifications, cohomologies},
pdfsubject={math},
pdfauthor={S. Feklistov},
pdfkeywords={Hartogs},
}

\begin{document}
\maketitle

\begin{abstract}
In these (not-completed) notes, we study the Hartogs extension phenomenon for holomorphic sections of holomorphic vector bundles over complex analytic varieties. Namely, we study properties of the Hartogs extension phenomenon with respect to the open embeddings, proper maps, compactifications, relations with the compact supports first cohomology, and the Lefschetz type property for sections of sheaves. As an application, we get a convex-geometric criterion of the Hartogs phenomenon for complex almost homogeneous algebraic G-varieties, where G is a semiabelian variety.
\end{abstract}


\tableofcontents

\newpage

 \section{Introduction}\label{sec1}

The classical Hartogs extension theorem states that for every domain $W\subset\mathbb{C}^{n}(n>1)$ and a compact set $K\subset W$ such that $W\setminus K$ is a connected set, the restriction homomorphism $\Gamma(W,\mathcal{O})\to \Gamma(W\setminus K, \mathcal{O})$ is an isomorphism.

A natural question arises: is this true for general complex analytic spaces? This phenomenon has been extensively studied in many situations, including Stein manifolds and spaces, $(n-1)$-complete normal complex spaces, and so on (see, for instance, \cite{Andersson,AndrHill,BanStan,ColtRupp,Harvey,Fek1,Fek2,Dwilewicz,Merker,Vassiliadou,Viorel}). 

In this paper, we use classical homological algebra methods (see, for instance, \cite{Manin, Shapira}). For instance, we often use \cite[Propositions 1.5.6, 1.8.7, 1.8.8.]{Shapira}, facts from \cite[Chapter II]{Shapira}, and the diagram chasing.  

In Section \ref{sec2}, we introduce the notion of the Hartogs sheaves (Definition \ref{defHart}). It is convenient to give this in more general settings. Using the Andreotti-Hill trick \cite[corollary 4.3]{AndrHill}, we get for the case of normal complex analytic varieties that the finiteness of $H^{1}_{c}(X,\mathcal{O}_{X})$ implies that $\mathcal{O}_{X}$ is Hartogs (Proposition \ref{lemmonholom}). For any locally free $\mathcal{O}_{X}$-module of finite rank $\mathcal{F}$, we obtain that the vanishing of $H^{1}_{c}(X,\mathcal{F})$ implies that $\mathcal{F}$ is Hartogs (Lemma \ref{mainlemma1}). Note that the Andreotti-Hill trick does not apply to arbitrary locally free $\mathcal{O}_{X}$-modules of finite rank. So, we have a natural question: is it true that the condition $\dim_{\mathbb{C}}H^{1}_{c}(X,\mathcal{F})<\infty$ implies that $\mathcal{F}$ is Hartogs? This problem can be reduced to the case of holomorphic line bundles associated with effective Cartier divisors (Section \ref{sectionreduxtoline}). We solved this problem (Theorem \ref{mainthsheaf}) for the case of complex analytic manifolds $X$ with the following additional assumption: $X\setminus K$ admits non-constant holomorphic functions for a sufficiently large compact set $K\subset X$, and each irreducible divisor of $X$ with compact support is $\mathcal{F}$-removable (Definition \ref{defi22}). For example, each exceptional irreducible divisor is $\mathcal{F}$-removable for any locally free $\mathcal{O}_{X}$-module of finite rank $\mathcal{F}$ (Example \ref{removexample}).

With some additional conditions, we may prove that the sheaf is Hartogs if and only if it is Hartogs over each domain (Corollary \ref{corth4}, Proposition \ref{prop}). If a complex analytic variety has only one topological end, then we obtain the necessary and sufficient conditions of the Hartogs property (Theorem \ref{maintheorem1}, Theorem \ref{th1'}). For example, each locally free $\mathcal{O}_{X}$-module of finite rank over a cohomological $(n-1)$-complete complex manifold is Hartogs (Example \ref{nminus1convexexamplehartg}), and for any domain $X$ with connected boundary of a complex manifold $X'$ with $H^{1}(X',\mathcal{O}_{X'})=0$, we have that $H^{1}_{c}(X,\mathcal{O}_{X})=0$ implies that $\mathcal{O}_{X}$ is Hartogs (Example \ref{examp}).

The Hartogs property has a nice behavior with respect to the proper surjective holomorphic maps (Propositions \ref{prop1}, \ref{proposition24}, \ref{proposition25}). In particular, if $X$ is a holomorphically convex noncompact normal complex analytic variety, $R\colon X\to Y$ is the corresponding Remmert reduction, and $\mathcal{F}\in Vect(\mathcal{O}_{X})$, then $\mathcal{F}$ is Hartogs if and only if $\dim Y>1$ (Corollary \ref{holoconvex}). More examples of noncompact holomorphically convex varieties can be obtained the following way: let $X'$ be a normal complex projective variety, $D$ be an effective basepoint-free Cartier divisor with connected support $\operatorname{Supp}(D)$, then $X'\setminus \operatorname{Supp}(D)$ is holomorphically convex; moreover, the Remmert reduction induced by the holomorphic map associated with the complete linear system $|D|$ (Example \ref{exlinearsys}).  

In Section \ref{sec3}, we consider very specific open embeddings. Namely, we consider $(b,\sigma)$-compactified pairs (Definition \ref{Defibsigma}). Roughly speaking, the pair $(X,\mathcal{F})$ (here $X$ is a complex analytic variety, $\mathcal{F}$ is a sheaf of $\mathbb{C}$-vector spaces) is $(b,\sigma)$-compactified if $X$ admits a nice compactification $X'$ such that $X'\setminus X$ is a proper analytic set which has only $b$ connected components, and the sheaf $\mathcal{F}$ extends to a sheaf $\mathcal{F}'$ over $X'$ such that $\dim_{\mathbb{C}} H^{1}(X',\mathcal{F}')=\sigma$. For pairs of the form $(X,\mathcal{O}_{X})$, we have some facts on numbers $(b,\sigma)$. First, by the Nagata theorem, each complex algebraic variety admits a compactification in the category of complex algebraic varieties. It follows that each complex algebraic variety is $(b,\sigma)$-compactifiable for some $(b,\sigma)$. Now if $X$ is a complex algebraic manifold (i.e., nonsingular complex algebraic variety), then the pair $(b,\sigma)$ does not depend on any compactification in the category of complex analytic manifolds. Namely, $b$ is exactly the number of topological ends of $X$, and $\sigma$ is exactly the dimension of the Albanese variety of $X$ (Proposition \ref{prop27}). In the case of almost homogeneous algebraic $G$-manifolds, $\sigma$ is exactly the dimension of the Albanese variety of the open $G$-orbit (Remark \ref{remar12}). If $G$ is a connected complex linear algebraic group, then $\sigma=0$, and the Sumihiro theorem implies that each complex algebraic $G$-manifold admits a $G$-equivariant compactification; hence $X$ is $(b,0)$-compactifiable if and only if $X$ has only $b$ topological ends (Remark \ref{remar12}).

In Section \ref{sec4}, we study the Hartogs property for $(1,\sigma)$-compactified pairs (Theorem \ref{thm30}) and its relation with the Lefschetz property (Corollary \ref{lef}, Remark \ref{lefrem}). If the analytic set $X'\setminus X$ is the support of a nef divisor, then we obtain a characterization of the Hartogs property in terms of the Iitaka and numerical dimensions of this divisor (Corollary \ref{nefcase}).

In Section \ref{sec5}, we consider the case of almost homogeneous algebraic $G$-varieties and an example. We describe the space of germs of holomorphic functions at infinity analytic set (more precisely, the space $\Gamma(Z,\mathcal{O}_{X'}|_{Z})$, where $Z=X'\setminus X$ for some $G$-equivariant compactification $X'$ of $X$). If $G$ is a semiabelian variety (i.e., $G$ is an extension of an abelian variety by an algebraic torus), then we describe the space $\Gamma(Z,\mathcal{O}_{X'}|_{Z})$ in a more precise way. In particular, this implies the convex-geometric criterion of the Hartogs phenomenon (Theorem \ref{semitoric}).

\textbf{Some notation.}
\begin{itemize} 
\item $Sh(X)$ is the category of sheaves of abelian groups over topological space $X$. 
\item $Sh_{\mathbb{C}}(X)$ is the category of sheaves of $\mathbb{C}$-vector spaces over topological space $X$.
\item $Mod(\mathbb{C})$ is the category of $\mathbb{C}$-vector spaces.
\item $Coh(\mathcal{O}_{X})$ is the category of coherent $\mathcal{O}_{X}$-modules over complex analytic space $X$.
\item $Vect(\mathcal{O}_{X})$ is the category of locally free $\mathcal{O}_{X}$-modules of finite ranks over complex analytic space $X$.
\item $Top$ is the category of topological spaces.
\item $An$ is the category of complex analytic varieties (objects: reducible, irreducible, and countable at infinity complex analytic spaces; morphisms: holomorphic maps).
\item $An_{norm}$ is the category of normal complex analytic varieties (it is a full subcategory of $An$).
\item $An_{sm}$ is the category of complex analytic manifolds (it is a full subcategory of $An$).
\item $Al^{a}$ is the analytification of the category of complex algebraic varieties (objects: complex algebraic varieties considered as complex analytic varieties; morphisms: morphisms of algebraic varieties considered as holomorphic maps).
\item $Al^{a}_{norm}$ is the analytification of the category of normal complex algebraic varieties.
\item $Al^{a}_{sm}$ is the analytification of the category of nonsingular complex algebraic varieties. 
\item $An_{G}$ is the category of complex analytic $G$-varieties (here $G$ is a complex Lie group acting holomorphically). 
\item $Al^{a}_{G}$ is the analytification of the category of complex algebraic $G$-varieties (here $G$ is a complex algebraic group acting algebraically).
\item $Al^{a}_{G,norm}$ is the analytification of the category of normal complex algebraic $G$-varieties.
\item $Al^{a}_{G,sm}$ is the analytification of the category of nonsingular complex algebraic $G$-varieties. 
\item If $T\colon \mathcal{A}\to \mathcal{B}$ is a functor between abelian categories, then we denote by 
\begin{gather*}
\mathbf{R}T\colon D^{+}(\mathcal{A})\to D^{+}(\mathcal{B})
\end{gather*} the functor between corresponding derived categories.
\item If $Z\subset X$ is a closed subset of the topological space $X$, $i\colon Z\hookrightarrow X$ is the canonical closed immersion, and $\mathcal{F}$ is a sheaf over $X$, then the inverse image sheaf $i^{-1}\mathcal{F}$ is often denoted by $\mathcal{F}|_{Z}$.
\end{itemize}

\newpage
\section{Historical review}

To be completed

\newpage
\section{Hartogs extension phenomenon}\label{sec2}
Let $X$ be a connected locally compact Hausdorff topological space (briefly, \textbf{CLCH-space}), and let $Sh_{\mathbb{C}}(X)$ be the category of sheaves of $\mathbb{C}$-vector spaces over $X$, and let $Mod(\mathbb{C})$ be the category of $\mathbb{C}$-vector spaces. Let $K\subset X$ be a compact set. We have the following left exact functors: 
\begin{gather*}
\Gamma_{K}\colon Sh_{\mathbb{C}}(X)\to Sh_{\mathbb{C}}(X)\quad\mathcal{F}\to \Gamma_{K}(\mathcal{F})=\{U\to \Gamma_{K\cap U}(U,\mathcal{F})\} \\
\Gamma_{X\setminus K}\colon Sh_{\mathbb{C}}(X)\to Sh_{\mathbb{C}}(X)\quad\mathcal{F}\to \Gamma_{X\setminus K}(\mathcal{F})=\{U\to \Gamma_{(X\setminus K)\cap U}(U,\mathcal{F})\},
\end{gather*}
and the canonical morphisms of functors: 
\begin{gather*}
\Gamma_{K}\rightarrow 1_{Sh_{\mathbb{C}}(X)}\\ 
1_{Sh_{\mathbb{C}}(X)}\rightarrow \Gamma_{X\setminus K}.
\end{gather*}
\begin{remark}\label{remiK}\noindent
\begin{itemize}
\item Let $i_{K}\colon X\setminus K\hookrightarrow X$ be the canonical open immersion. We obtain the following functors: 
\begin{gather*}
(i_{K})^{-1}\colon Sh_{\mathbb{C}}(X)\to Sh_{\mathbb{C}}(X\setminus K)\\ 
(i_{K})_{*}\colon Sh_{\mathbb{C}}(X\setminus K)\to Sh_{\mathbb{C}}(X).
\end{gather*}
We have $\Gamma_{X\setminus K}=(i_{K})_{*}\circ (i_{K})^{-1}$.
\item Let $j_{K}\colon K\hookrightarrow X$ be the canonical closed immersion. We obtain the following functors: 
\begin{gather*}
(j_{K})^{-1}, (j_{K})^{!}\colon Sh_{\mathbb{C}}(X)\to Sh_{\mathbb{C}}(K).
\end{gather*}

We have $(j_{K})^{!}=(j_{K})^{-1}\circ \Gamma_{K}$  
\end{itemize}
\end{remark}
Let $\Gamma(X,-)\colon Sh_{\mathbb{C}}(X)\to Mod(\mathbb{C})$ be the global sections functor. We have the following canonical morphisms of the left exact functors: 
\begin{gather*}
\Gamma(X,-)\circ \Gamma_{K}\rightarrow\Gamma(X,-)\\ 
\Gamma(X,-) \rightarrow \Gamma(X,-)\circ \Gamma_{X\setminus K}.
\end{gather*}

Let us remark that $\Gamma(X,-)\circ \Gamma_{X\setminus K}=\Gamma_{X\setminus K}(X,-)$ and $\Gamma(X,-)\circ \Gamma_{K}=\Gamma_{K}(X,-)$. For any injective object $I\in Sh_{\mathbb{C}}(X)$, we have the following exact sequence: 
\begin{gather*}
0\to\Gamma_{K}(X,I)\to \Gamma(X,I) \to \Gamma_{X\setminus K}(X,I)\to 0.
\end{gather*}

It follows that for any sheaf $\mathcal{F}\in Sh_{\mathbb{C}}(X)$ we obtain the following distinguished triangle: 
\begin{equation}\label{eq2}
\mathbf{R}\Gamma_{K}(X,\mathcal{F})\to \mathbf{R}\Gamma(X,\mathcal{F})\to \mathbf{R}\Gamma_{X\setminus K}(X,\mathcal{F})\to_{+1}
\end{equation}

\begin{definition}\label{defHart}
Let $X$ be a CLCH-space, $\mathcal{F}\in Sh_{\mathbb{C}}(X)$, $K\subset X$ be a compact set.

\begin{enumerate}
\item The pair $(K,X)$ is called a \textbf{Hartogs pair}, if $X\setminus K$ is a connected set. 
\item Let $(K,X)$ is a Hartogs pair. We say that the sheaf $\mathcal{F}$ \textbf{is Hartogs w.r.t. $(K,X)$}, if the canonical morphism 
$\Gamma(X,\mathcal{F})\to \Gamma(X\setminus K,\mathcal{F})$ is an epimorphism. 
\item We say that the sheaf $\mathcal{F}$ \textbf{is Hartogs} if $\mathcal{F}$ is Hartogs with respect to every Hartogs pair $(K,X)$ 
\end{enumerate}
\end{definition}

A complex analytic space is called \textbf{complex analytic variety} if it is reducible, irreducible, and countable at infinity. 

\begin{example}
Note that if $X$ is a compact analytic space and $\mathcal{F}$ is a coherent $\mathcal{O}_{X}$-module, then $\mathcal{F}$ is not Hartogs. Let $(K,X)$ be a Hartogs pair such that $U=X\setminus K$ is a Stein space. The Cartan theorem A implies that $\dim_{\mathbb{C}} \Gamma (U,\mathcal{F})=\infty$, but the Cartan-Serre theorem implies that $\dim_{\mathbb{C}}\Gamma(X,\mathcal{F})<\infty$. 
\end{example}

Note that a noncompact CLCH-space $X$ has only one topological end if and only if for any compact set $K\subset X$ the pair $(\mu(K),X)$ is a Hartogs, where $\mu(K)$ is the union of $K$ with all connected components of $X\setminus K$ that are relatively compact in $X$ (see \cite{Freudenthal, Peschke} for more details about topological ends).

\begin{remark}\label{topologicalend}
Let us remember on the topological ends. Let $X$ be a topological space, and suppose that 
\begin{gather*}
K_{1}\subset K_{2}\subset\cdots \subset K_{n}\subset \cdots 
\end{gather*} is an ascending sequence of compact subsets of X whose interiors cover $X$ (that is, X is hemicompact). 
The complementary sets $C_{n}:= X\setminus K_{n}$ form an inverse sequence:
\begin{gather*}
C_{1}\supset C_{2}\supset\cdots \supset C_{n}\supset \cdots 
\end{gather*}
Define the \textbf{set of ends} $\mathcal{E}(X)$ as the collection of sequences $\{U_n\}$ of non-empty sets 
\begin{gather*}
U_{1}\supset U_{2}\supset\cdots \supset U_{n}\supset \cdots ,
\end{gather*}
where $U_{n}$ is a connected component of $C_n$. 

The number of ends does not depend on the specific sequence $\{K_{n}\}$ of compact sets and is denoted by $e(X)$. 

The definition of ends given above applies only to spaces X that possess an exhaustion by compact sets. Let $X$ be an arbitrary topological space (which is not necessarily hemicompact). Consider the direct system $\{K\}$ of compact subsets of $X$ and inclusion maps. There is a corresponding inverse system $\{\pi_{0}(X\setminus K)\}$, where $\pi_{0}(Y)$ denotes the set of connected components of a space $Y$, and each inclusion map $Y\hookrightarrow Z$ induces a function $\pi_{0}(Y) \to \pi_{0}(Z)$. Define the \textbf{set of ends} $\mathcal{E}(X)$ as the inverse limit of this inverse system. 
\end{remark}

Let $\Gamma_{c}(X,-)\colon Sh_{\mathbb{C}}(X)\to Mod(\mathbb{C})$ be the global sections functor with compact supports. Define \textbf{the sections at boundary functor} by the rule: 
\begin{gather*}
\Gamma (\partial X,-)\colon Sh_{\mathbb{C}}(X)\to Mod(\mathbb{C})\quad \mathcal{F}\to \Gamma(\partial X,\mathcal{F}):=\varinjlim\limits_{K\subset X}\Gamma_{X\setminus K}(X,\mathcal{F}).
\end{gather*}

Taking the direct limit over compact sets in (\ref{eq2}), we obtain 
\begin{equation}\label{eq3'}
\mathbf{R}\Gamma_{c}(X,\mathcal{F})\cong\varinjlim\limits_{K} \mathbf{R}\Gamma_{K}(X,\mathcal{F})
\end{equation}
and the following distinguished triangle: 
\begin{equation}\label{eq3}
\mathbf{R}\Gamma_{c}(X,\mathcal{F})\to \mathbf{R}\Gamma(X,\mathcal{F})\to \mathbf{R}\Gamma(\partial X,\mathcal{F})\to_{+1}
\end{equation}

Moreover, we have the following commutative diagram of distinguished triangles:

\begin{equation}\label{eq3''}
\begin{diagram} 
\node{\mathbf{R}\Gamma_{K}(X,\mathcal{F})} \arrow{e,t}{} \arrow{s,r}{}
\node{\mathbf{R}\Gamma(X,\mathcal{F})} \arrow{e,t}{} \arrow{s,r}{id} 
\node{\mathbf{R}\Gamma_{X\setminus K}(X,\mathcal{F})\to_{+1}} \arrow{s,r}{} 
\\
\node{\mathbf{R}\Gamma_{c}(X,\mathcal{F})} \arrow{e,t}{} 
\node{\mathbf{R}\Gamma(X,\mathcal{F})} \arrow{e,t}{} 
\node{\mathbf{R}\Gamma(\partial X,\mathcal{F})\to_{+1}}
\end{diagram}
\end{equation}

We only need the global sections of sheaves. We obviously obtain the following lemma, which follows from (\ref{eq3}), (\ref{eq3'}), (\ref{eq3''}).

\begin{lemma} \label{mainlemma-1}
\begin{enumerate}
\item Let $X$ be a noncompact CLCH-space and $\mathcal{F}\in Sh_{\mathbb{C}}(X)$. If $H^{1}_{c}(X,\mathcal{F})=0$, then the canonical morphism $\Gamma(X,\mathcal{F})\to \Gamma(\partial X,\mathcal{F})$ is an epimorphism.
\item Let $X$ be a noncompact CLCH-space and $\mathcal{F}\in Sh_{\mathbb{C}}(X)$. If for any compact set $K\subset X$ there exists a compact set $K'\subset X$ such that $K\subset K'$ and $H^{1}_{K'}(X,\mathcal{F})=0$, then $H^{1}_{c}(X,\mathcal{F})=0$. 
\item Let $X$ be a noncompact CLCH-space and $\mathcal{F}\in Sh_{\mathbb{C}}(X)$. Let $(K,X)$ be a Hartogs pair. Suppose the canonical morphism $\Gamma(X,\mathcal{F})\to \Gamma(\partial X,\mathcal{F})$ is an epimorphism and the canonical morphism 
$\Gamma(X\setminus K,\mathcal{F})\to \Gamma(\partial X,\mathcal{F})$ is a monomorphism. Then $\mathcal{F}$ is Hartogs w.r.t $(K,X)$.
\end{enumerate}
\end{lemma}

\begin{remark}\label{identlemma}
The identity theorem for locally free $\mathcal{O}_{X}$-modules of finite rank: let $\mathcal{F}$ be a locally free $\mathcal{O}_{X}$-module of finite rank over a complex analytic variety $X$ and $f\in \Gamma(X,\mathcal{F})$; if there exists an open subset $U\subset X$ such that $f|_{U}=0$, then $f=0$. 

Assume that $\mathcal{F}$ is an arbitrary coherent $\mathcal{O}_{X}$-module. There exists an analytic thin subset $S\subset X$ such that $\mathcal{F}|_{X\setminus S}$ is a locally free $\mathcal{O}_{X\setminus S}$-module of finite rank \cite[Chapter 1, Section 7]{Grauert}. If the restriction homomorphism $\Gamma(X,\mathcal{F})\to \Gamma (X\setminus S, \mathcal{F})$ is a monomorphism, then the identity theorem is also true. 

Define $S_{m}(\mathcal{F}):=\{x\in X\mid \operatorname{codh}\mathcal{F}_{x}\leq m\}$. Note that $S=S_{n-1}(\mathcal{F})$, $S_{i}(\mathcal{F})\subset S_{i+1}(\mathcal{F})$. If $\dim S_{i+1}(\mathcal{F})<i+1$ for all $i$ (in particular, $\dim S<n-1$), then the restriction homomorphism $\Gamma(X,\mathcal{F})\to \Gamma (X\setminus S, \mathcal{F})$ is a monomorphism \cite[Chapter II, Theorem 5.8]{Grauert}. This is also true for the reflexive sheaves over normal complex analytic varieties with $\dim S<n-1$ \cite[Chapter II, Theorem 5.29]{Grauert}.
\end{remark}

Now we give the sufficient condition of the Hartogs phenomenon for the locally free $\mathcal{O}_{X}$-modules of finite ranks.

\begin{lemma}\label{mainlemma1}

Let $X$ be a noncompact complex analytic variety, and let $\mathcal{F}\in Vect(\mathcal{O}_{X})$ be a locally free $\mathcal{O}_{X}$-module of finite rank. If $H^{1}_{c}(X,\mathcal{F})=0$, then the canonical morphism 
\begin{gather*}
\Gamma(X,\mathcal{F})\to \Gamma(\partial X,\mathcal{F})
\end{gather*} is an isomorphism. In particular, $\mathcal{F}$ is Hartogs and $X$ has only one topological end. 
\end{lemma}

\begin{proof}
By the identity theorem (see Remark \ref{identlemma}), $\Gamma(X\setminus K,\mathcal{F})\to \Gamma(\partial X,\mathcal{F})$ is a monomorphism for any Hartogs pair $(K,X)$. Finally, we use lemma \ref{mainlemma-1}. 
\end{proof}

\begin{remark}
Let $\mathcal{F}$ be a coherent $\mathcal{O}_{X}$-module and $S$ be the singular locus of $\mathcal{F}$. If the restriction homomorphism $\mathcal{F}(X)\to \mathcal{F}(X\setminus S)$ is an injective, then the Lemma \ref{mainlemma1} is also true for that $\mathcal{F}$ (see Remark \ref{identlemma}).
\end{remark}

\begin{example}\label{nminus1convexexamplehartg}
Let $X$ be a cohomological $(n-1)$-complete complex manifold, $n>1$, and $\mathcal{F}$ be a locally free $\mathcal{O}_{X}$-module of finite rank. By \cite[Proposition 1]{Viorel2}, it follows that $ H^{1}_{c}(X,\mathcal{F})=0$. In this case, $\mathcal{F}$ is Hartogs. 
\end{example}

In the case of normal complex analytic spaces and the structure sheaves, we only require the finite dimension of the space $H^{1}_{c}(X,\mathcal{O}_{X})$ for the Hartogs phenomenon. 

\begin{proposition}\label{lemmonholom}
Let $X$ be a noncompact normal complex analytic variety that has only one topological end. If $m:=\dim_{\mathbb{C}} H^{1}_{c}(X,\mathcal{O}_{X})<\infty$, then the canonical homomorphism 
\begin{gather*}
\Gamma(X,\mathcal{O}_{X})\to \Gamma(\partial X,\mathcal{O}_{X})
\end{gather*} is an isomorphism. In particular, the sheaf $\mathcal{O}_{X}$ is Hartogs.
\end{proposition}
\begin{proof}
We have the following long exact sequence: 

\begin{equation*}
	\xymatrix@C=0.5cm{
0 \ar[r] &  \Gamma(X,\mathcal{O}_{X}) \ar[rr]^{r} && \Gamma(\partial X,\mathcal{O}_{X}) \ar[rr]^{c} && H^{1}_{c}(X,\mathcal{O}_{X}) \ar[r] & \cdots }
\end{equation*}

Now we use the Andreotti-Hill trick (see \cite[corollary 4.3]{AndrHill}). We may assume that $\Gamma(\partial X,\mathcal{O}_{X})\neq\mathbb{C}$. Consider an equivalence class $f\in \Gamma(\partial X,\mathcal{O}_{X})$ of a nonconstant holomorphic function on $X\setminus K$.

We may assume that $c(f^{i})$ are non-zero for any $1\leq i\leq m+1$ replacing $f$ by $af+b$ for some $a,b\in\mathbb{C}\setminus \{0\}$. So, the elements $c(f), c(f^{2}),\cdots, c(f^{m+1})$ are non-zero and linearly dependent. This means that there exists a non-zero polynomial $P\in \mathbb{C}[T]$ of degree $m+1$ such that $c(P(f))=0$. It follows that there is a non-constant holomorphic function $H\in\Gamma(Y,\mathcal{O})$ such that $r(H)=P(f)$.

Now the elements $ c(f),c(r(H)f),\cdots, c(r(H)^{m}f)$ are non-zero and linearly dependent. Hence there exists a polynomial $P_{1}\in \mathbb{C}[T]$ such that $c(P_{1}(r(H))f)=0$. It follows that there exists a holomorphic function $F\in\Gamma(Y,\mathcal{O})$ such that $r(F)=P_{1}(r(H))f$. Denoting $G=P_{1}(H)$, we obtain $r(F)=r(G)f$.

Since $r(G^{m+1}H)=r(G^{m+1}P(F/G))$, it follows that $G^{m+1}H=G^{m+1}P(F/G)$. Hence, we obtain $H=P(F/G)$ on $Y\setminus \{G=0\}$. It follows that $F/G\in \Gamma(Y\setminus \{G=0\}, \mathcal{O})$ and is locally bounded on $Y\setminus \{G=0\}$. Since $G\neq 0$, the Riemann extension theorem implies that $F/G\in \Gamma(Y,\mathcal{O})$ and $r(F/G)=f$. Therefore, the canonical homomorphism $r$ is isomorphic. The last statement follows from Lemma \ref{mainlemma-1}.
\end{proof}

\begin{remark}\label{remarkoncoher1}\noindent
\begin{enumerate}
\item Note that a noncompact CLCH-space $X$ has at most $\dim_{\mathbb{C}} H^{1}_{c}(X,\mathcal{O}_{X})$ topological ends. 
\item Proposition \ref{lemmonholom} is still true if we replace $\mathcal{O}_{X}$ by a sheaf of ideals of $\mathcal{O}_{X}$ or by a free $\mathcal{O}_{X}$-modules. 
\item We may use the Andreotti-Hill trick to the following long exact sequence 
\begin{gather*}
0\to\Gamma(X, \mathcal{O}_{X}) \to \Gamma(X\setminus K, \mathcal{O}_{X}) \to H^{1}_{K}(X,\mathcal{O}_{X}) \to \cdots .
\end{gather*} Hence if $X$ is a noncompact normal complex analytic variety and $(K, X)$ is a Hartogs pair, then the condition $\dim_{\mathbb{C}} H^{1}_{K}(X,\mathcal{O}_{X})<\infty$ implies that the sheaf $\mathcal{O}_{X}$ is Hartogs w.r.t. $(K,X)$.
\item The Andreotti-Hill trick may be applied to the sheaf of meromorphic function $\mathcal{M}$ over a noncompact complex analytic variety (which is not necessarily normal). Note that the identity theorem is also true for the meromorphic functions. It follows that Lemma \ref{mainlemma1} is true for this case. 
\end{enumerate}
\end{remark}

\textbf{Question: } Let $X$ be a noncompact normal complex analytic variety which has only one topological end, and let $\mathcal{F}$ be a locally free $\mathcal{O}_{X}$-module of finite rank. Is it true that the condition $\dim_{\mathbb{C}}H^{1}_{c}(X,\mathcal{F})<\infty$ implies that $\mathcal{F}$ is Hartogs? 

This problem can be reduced to the case of holomorphic line bundles associated with effective Cartier divisors (see section \ref{sectionreduxtoline}). We solved this problem for the case of complex analytic manifolds $X$ with the following additional assumption: $X\setminus K$ have non-constant holomorphic functions for a sufficiently large compact set $K\subset X$ and any irreducible divisor $D$ of $X$ with compact support if $\mathcal{F}$-removable (i.e., the restriction homomorphism 
$\Gamma(X,\mathcal{F})\to \Gamma(X\setminus D,\mathcal{F})$ is an epimorphism).

\subsection{Hartogs and open embeddings}

Let $X$ be a CLCH-space. For any domain $W\subset X$ (i.e., open and connected subset of $X$) and for any compact subset $K\subset W$, there exists the following functor: 
\begin{gather*}
\Gamma_{X\setminus K}(X,-)\times_{\Gamma_{W\setminus K}(X,-)}\Gamma_{W}(X,-)\colon Sh_{\mathbb{C}}(X)\to Mod(\mathbb{C}).
\end{gather*}

\begin{remark}\label{mainlemma2}
If $X$ is a connected topological space, $W\subset X$ is a domain, and $(K, W)$ is a Hartogs pair, then $(K, X)$ is a Hartogs pair. 
\end{remark}

Since $W\setminus K=(X\setminus K) \cap W$, it follows that the following natural isomorphism of functors: 
\begin{gather*} 
\Gamma(X,-)= \Gamma_{X\setminus K}(X,-)\times_{\Gamma_{W\setminus K}(X,-)}\Gamma_{W}(X,-).
\end{gather*}

In particular, we obtain the following easy lemma. 

\begin{lemma}\label{corth5}
Let $X$ be a CLCH-space, $\mathcal{F}\in Sh_{\mathbb{C}}(X)$, $W\subset X$ be a domain, and $(K, W)$ be a Hartogs pair. If $\mathcal{F}|_{W}$ is Hartogs w.r.t. $(K,W)$, then $\mathcal{F}$ is Hartogs w.r.t. $(K,X)$. 
\end{lemma}

With additional conditions, we obtain a converse statement to Lemma \ref{corth5}. First, for the sheaf $\mathcal{F}\in Sh_{\mathbb{C}}(X)$, we define the irregularity of $\mathcal{F}$ as $\sigma_{1}(X,\mathcal{F}):=\dim_{\mathbb{C}}H^{1}(X,\mathcal{F}).$ Let us define the irregularity of X as the number $\sigma_{1}(X,\mathcal{O}_{X})$.

\begin{example}\noindent
\begin{enumerate} 
\item $\sigma_{1}(X, \mathcal{O}_{X})=0$ for projective spaces $\mathbb{CP}^{n}$, flag varieties (generally, for spherical varieties \cite[Corollaire 1]{Brion1}).
\item $\sigma_{1}(X, \mathcal{O}_{X})\neq 0$ for complex tori (in particular, for abelian varieties). If $X=\mathbb{C}^{n}/\Lambda$, then $\sigma_{1}(X, \mathcal{O}_{X})=n$.
\item If $X=\mathbb{C}\mathbb{P}^{1}, \mathcal{F}=\mathcal{O}_{\mathbb{CP}^{1}}(n)$, then $\sigma_{1}(X, \mathcal{F})=\begin{cases}
0, & \text{if $n\geq -1$;} \\
-n-1, & \text{if $n<-1$.}
\end{cases}$. Actually, for the canonical sheaf $K_{\mathbb{C}\mathbb{P}^{1}}$ we have $K_{\mathbb{CP}^{1}}\cong \mathcal{O}_{\mathbb{CP}^{1}}(-2)$. The Serre duality implies that $\sigma_{1}(X, \mathcal{O}_{\mathbb{CP}^{1}}(n))=\dim_{\mathbb{C}}(H^{0}(\mathbb{CP}^{1}, \mathcal{O}_{\mathbb{CP}^{1}}(-n-2))).$ 
\end{enumerate}
\end{example}

\begin{lemma}\label{corth4}
Let $X$ be a CLCH-space, $\mathcal{F}\in Sh_{\mathbb{C}}(X)$, $W\subset X$ be a domain, and $(K, W)$ be a Hartogs pair. Suppose there exists a CLCH-space $X'$, a sheaf $\mathcal{F}'\in Sh_{\mathbb{C}}(X')$ with $\sigma_{1}(X',\mathcal{F}')=0$, and an open embedding $i\colon X\hookrightarrow X'$ such that $\mathcal{F}=i^{-1}\mathcal{F}'$. If $\mathcal{F}$ is Hartogs w.r.t. $(K,X)$, then $\mathcal{F}|_{W}$ is Hartogs w.r.t. $(K,W)$.
\end{lemma}
 
\begin{proof}
The proof follows from the following commutative diagram: 

\[\begin{diagram} 
\node{\Gamma(X',\mathcal{F}')}\arrow{s,r}{}\arrow{e,r}{} \node{\Gamma(X'\setminus K,\mathcal{F}')} \arrow{s,r}{} \arrow{e,r}{}\node{H^{1}_{K}(X',\mathcal{F}')} \arrow{s,r}{\cong} \arrow{e,r}{}\node{0}
\\
\node{\Gamma(X,\mathcal{F})} \arrow{e,t}{} \arrow{s,r}{}
\node{\Gamma(X\setminus K,\mathcal{F})} \arrow{s,r}{} \arrow{e,r}{}\node{H^{1}_{K}(X,\mathcal{F})} \arrow{s,r}{\cong}
\\
\node{\Gamma(W,\mathcal{F})} \arrow{e,t}{} 
\node{\Gamma(W\setminus K,\mathcal{F})} \arrow{e,r}{}\node{H^{1}_{K}(W,\mathcal{F})} 
\end{diagram}\]
\end{proof} 

In the case of the structure sheaf of a noncompact normal complex analytic variety, we obtain the following results. 

\begin{proposition}\label{prop}
Let $X$ be a noncompact normal complex analytic variety, $W\subset X$ be a domain, and $(K, W)$ be a Hartogs pair. Suppose there exists a CLCH-space $X'$, a sheaf $\mathcal{F}'\in Sh_{\mathbb{C}}(X')$ with $\sigma_{1}(X',\mathcal{F}')<\infty$, and an open embedding $i\colon X\hookrightarrow X'$ such that $\mathcal{O}_{X}=i^{-1}\mathcal{F}'$. If $\mathcal{O}_{X}$ is Hartogs w.r.t. $(K,X)$, then $\mathcal{O}_{X}|_{W}$ is Hartogs  w.r.t. $(K,W)$.
\end{proposition}

\begin{proof}
The proof follows from Lemma \ref{corth5}, from the following commutative diagram 

\[
\begin{diagram} 
\node{\Gamma(X', \mathcal{F}')} \arrow{e,t}{} \arrow{s,r}{}
\node{\Gamma(X'\setminus K, \mathcal{F}')} \arrow{e,t}{} \arrow{s,r}{}
\node{H^{1}_{K}(X',\mathcal{F}')} \arrow{e,t}{} \arrow{s,r}{\cong} 
\node{H^{1}(X',\mathcal{F}')}
\\
\node{\Gamma(X, \mathcal{O}_{X})} \arrow{e,t}{} \arrow{s,r}{}
\node{\Gamma(X\setminus K, \mathcal{O}_{X})} \arrow{e,t}{} \arrow{s,r}{}
\node{H^{1}_{K}(X,\mathcal{O}_{X})} \arrow{s,r}{\cong} 
\\
\node{\Gamma(W, \mathcal{O}_{X})} \arrow{e,t}{} 
\node{\Gamma(W\setminus K, \mathcal{O}_{X})} \arrow{e,t}{} 
\node{H^{1}_{K}(W,\mathcal{O}_{X})} 
\end{diagram}
\]
and from the Andreotti-Hill trick applying to the last line of this diagram.
\end{proof}

If a CLCH-space has only one topological end, then we have the following statement. 

\begin{lemma}\label{mainprop'}
Let $X$ be a noncompact CLCH-space which has only one topological end, and let $\mathcal{F}\in Sh_{\mathbb{C}}(X)$ be a sheaf of $\mathbb{C}$-vector spaces. Suppose there exists a CLCH-space $X'$, a sheaf $\mathcal{F}'\in Sh_{\mathbb{C}}(X')$ with $\sigma_{1}(X',\mathcal{F}')<\infty$, and an open embedding $i\colon X\hookrightarrow X'$ such that $\mathcal{F}=i^{-1}\mathcal{F}'$. If $\mathcal{F}$ is Hartogs, then $\dim_{\mathbb{C}} H^{1}_{c}(X,\mathcal{F})\leq \sigma_{1}(X',\mathcal{F}') $. 
\end{lemma}

\begin{proof}
The proof follows from the following commutative diagram:
\[
\begin{diagram} 
\node{\Gamma(X',\mathcal{F}')}\arrow{s,r}{}\arrow{e,r}{} \node{\Gamma(X'\setminus K,\mathcal{F}')} \arrow{s,r}{} \arrow{e,r}{}\node{H^{1}_{K}(X',\mathcal{F}')} \arrow{s,r}{\cong} \arrow{e,r}{}\node{H^{1}(X',\mathcal{F}')}
\\
\node{\Gamma(X,\mathcal{F})} \arrow{e,t}{}
\node{\Gamma(X\setminus K,\mathcal{F})} \arrow{e,r}{}\node{H^{1}_{K}(X,\mathcal{F})}
\end{diagram}
\]
and from the canonical isomorphism $H^{1}_{c}(X,\mathcal{F})\cong\varinjlim\limits_{K} H^{1}_{K}(X,\mathcal{F})$ (here the colimit taking over compact sets $K$ such that $(K,X)$ is a Hartogs pair).
\end{proof}

Lemma \ref{mainlemma1}, \ref{corth4}, \ref{mainprop'} implies the following necessary and sufficient condition of the Hartogs phenomenon. 

\begin{theorem}\label{maintheorem1}
Let $X$ be a noncompact complex analytic variety which has only one topological end, and let $\mathcal{F}\in Vect(\mathcal{O}_{X})$ be a locally free $\mathcal{O}_{X}$-module of finite rank. Suppose there exists a CLCH-space $X'$, a sheaf $\mathcal{F}'\in Sh_{\mathbb{C}}(X')$ with $\sigma_{1}(X',\mathcal{F}')=0$, and an open embedding $i\colon X\hookrightarrow X'$ such that $\mathcal{F}=i^{-1}\mathcal{F}'$. Then the following conditions are equivalent:
\begin{enumerate}

\item $H^{1}_{c}(X,\mathcal{F})=0$

\item The sheaf $\mathcal{F}$ is Hartogs. 

\item The sheaf $\mathcal{F}|_{W}$ is Hartogs for any domain $W\subset X$.

\end{enumerate} 

\end{theorem}

\begin{example}\label{examp}\noindent
\begin{enumerate}
\item Let $X$ be a Stein variety, $\dim X\geq 2$, $\mathcal{F}\in Vect(\mathcal{O}_{X})$. By the Cartan theorem, it follows that $H^{1}(X,\mathcal{F})=0$. By \cite[Chapter VII, Section D, Theorem 2]{Ganning} or \cite[Chapter I, Corollary 4.10]{BanStan}, it follows that $X$ has only one topological end. Hence we may take $X'=X,\mathcal{F}'=\mathcal{F}$. It follows that $H^{1}_{c}(X,\mathcal{F})=0$ if and only if the sheaf $\mathcal{F}$ is Hartogs. Note that the implication "$\Leftarrow$" is also following from $\operatorname{codh}(\mathcal{F})=\dim X\geq 2$ \cite[Chapter 1, Theorem 3.6, Corollary 4.2]{BanStan}. 
\item Let $X$ be a domain with connected boundary of a complex manifold $X'$. Assume that $H^{1}(X',\mathcal{O}_{X'})=0$. It follows that $H^{1}_{c}(X,\mathcal{O}_{X})=0$ if and only if $\mathcal{O}_X$ is Hartogs. 
\item Let $X'$ be a compact complex manifold and $X:=X'\setminus\{pt\}$, where $pt$ is a point of $X'$. Let $\mathcal{F}$ be a locally free $\mathcal{O}_{X'}$-module of finite rank. Since $\dim_{\mathbb{C}} H^{1}_{c}(X,\mathcal{F})=\infty$, we obtain the sheaf $\mathcal{F}|_{X}$ is not Hartogs. Actually, we have the following exact sequence:
\begin{equation*}
	\xymatrix@C=0.5cm{
0 \ar[r] &  \Gamma(X', \mathcal{F}) \ar[rr] && \mathcal{F}_{pt} \ar[rr] && H^{1}_{c}(X,\mathcal{F}) \ar[r] & \cdots }
\end{equation*}
We see that $\dim_{\mathbb{C}} \Gamma(X', \mathcal{F})<\infty$ but $\dim_{\mathbb{C}} \mathcal{F}_{pt}=\infty$. 
\end{enumerate}
\end{example}

In the case of the structure sheaf of a noncompact normal complex analytic variety, we have more general results. 

\begin{theorem}\label{th1'}
Let $X$ be a noncompact normal complex analytic variety that has only one topological end. Suppose there exists a CLCH-space $X'$, a sheaf $\mathcal{F}'\in Sh_{\mathbb{C}}(X')$ with $\sigma_{1}(X',\mathcal{F}')<\infty$, and an open embedding $i\colon X\hookrightarrow X'$ such that $\mathcal{O}_{X}=i^{-1}\mathcal{F}'$. Then the following conditions are equivalent:
\begin{enumerate}
\item The sheaf $\mathcal{O}_{X}$ is Hartogs. 
\item The sheaf $\mathcal{O}_{X}|_{W}$ is Hartogs for any domain $W\subset X$.
\item $\dim_{\mathbb{C}} H^{1}_{c}(X,\mathcal{O}_{X})<\infty$.
\item $\dim_{\mathbb{C}} H^{1}_{c}(X,\mathcal{O}_{X})\leq \sigma(X',\mathcal{F}')$. 
\end{enumerate} 
\end{theorem}

\begin{proof}

The implications $2 \Rightarrow 1,  4\Rightarrow 3$ are clear. The implication $1 \Rightarrow 2$ follows from Proposition \ref{prop}. The implication $1 \Rightarrow 4$ follows from Lemma \ref{mainprop'}. The implication $3 \Rightarrow 1$ follows from Proposition \ref{lemmonholom}. 
\end{proof}

\begin{remark}
Let $X$ be a complex analytic variety, $\mathcal{F}\in Vect(\mathcal{O}_{X})$. Assume that $\dim_{\mathbb{C}} H^{1}_{c}(X,\mathcal{F})=\infty$. If $X$ has at least two topological ends, then $\mathcal{F}$ may or may not be a Hartogs sheaf. For instance, consider open submanifolds of $\mathbb{CP}^{n}$, $(n\geq 2)$, and let $\mathcal{F}\in Vect(\mathcal{O}_{\mathbb{CP}^{n}})$. In the following cases, we have $\dim_{\mathbb{C}} H^{1}_{c}(X,\mathcal{F})=\infty$:
\begin{enumerate}
\item Consider $X=\mathbb{CP}^{n}\setminus \{x,y\}$, where $x, y$ are distinct points of $\mathbb{CP}^{n}$. Then the sheaf $\mathcal{F}|_{X}$ is not Hartogs because there exists a compact set $K\subset \mathbb{CP}^{n}$ such that $y\not\in K$ and $\mathbb{CP}^{n}\setminus K$ is a Stein manifold. 
\item Consider $X=\mathbb{CP}^{n}\setminus \{x,H\}$, where $x$ is a point and $H$ is a hyperplane of $\mathbb{CP}^{n}$ such that $x\not\in H$. Since $X$ is an open submanifold of the Stein manifold $X\setminus H$, it follows that $\mathcal{F}|_{X}$ is Hartogs. 
\end{enumerate}
\end{remark}

\subsection{Hartogs and proper surjective holomorphic maps}

Let $f\colon X\to Y$ be a proper surjective continuous map between CLCH-spaces. We have the following commutative diagrams:

\[
\begin{diagram} 
\node{X\setminus K}\arrow[2]{s,r}{}\arrow[2]{e,r}{} \node[2]{X} \arrow[2]{s,r}{f} 
\\
\node[2]{X\setminus f^{-1}(f(K))} \arrow{ne,t}{} \arrow{sw,r}{}\arrow{nw,r}{}
\\
\node{Y\setminus f(K)} \arrow[2]{e,t}{}
\node[2]{Y} 
\end{diagram}\quad 
\begin{diagram} 
\node{K}\arrow{se,r}{}\arrow[2]{s,r}{}\arrow[2]{e,r}{} \node[2]{X} \arrow[2]{s,r}{f} 
\\
\node[2]{f^{-1}(f(K))} \arrow{ne,t}{} \arrow{sw,r}{}
\\
\node{f(K)} \arrow[2]{e,t}{}
\node[2]{Y} 
\end{diagram}
\]

We have the following natural isomorphism of functors:
\begin{gather*}
\Gamma(X,-)\cong\Gamma(Y,-)\circ f_{*}.
\end{gather*} This implies the following natural isomorphism: \begin{gather*}
\mathbf{R}\Gamma(X,-)\cong\mathbf{R}\Gamma(Y,-)\circ \mathbf{R}f_{*}.
\end{gather*} 

The natural morphism 
\begin{gather*} 
\Gamma_{K}\rightarrow \Gamma_{f^{-1}(f(K))}
\end{gather*} and the natural isomorphism 
\begin{gather*}
\Gamma_{f(K)}\circ f_{*} \cong f_{*}\circ \Gamma_{f^{-1}(f(K))}
\end{gather*} implies the following natural morphism:
\begin{multline*}
\Gamma_{K}(X,-)=\Gamma(X,-)\circ \Gamma_{K}\rightarrow \Gamma (X, -)\circ \Gamma_{f^{-1}(f(K))}=\\=\Gamma(Y,-)\circ f_{*}\circ \Gamma_{f^{-1}(f(K))}\cong \Gamma (Y, -)\circ \Gamma_{f(K)}\circ f_{*}=\Gamma_{f(K)}(Y,-)\circ f_{*}.
\end{multline*}

This implies the following natural morphism and isomorphism: \begin{gather*}
\mathbf{R}\Gamma_{K}(X,-)\to \mathbf{R}\Gamma_{f^{-1}(f(K))}(X,-)\cong\mathbf{R}\Gamma_{f(K)}(Y,-)\circ \mathbf{R}f_{*}.
\end{gather*} 

The same way we obtain that the natural morphism
\begin{gather*}
\Gamma_{X\setminus K} \rightarrow \Gamma_{X\setminus f^{-1}(f(K))}
\end{gather*} and the natural isomorphism 
\begin{gather*}
\Gamma_{Y\setminus f(K)}\circ f_{*} \cong f_{*}\circ \Gamma_{X\setminus f^{-1}(f(K))}
\end{gather*} implies the following natural morphism: 
\begin{multline*}
\Gamma_{X\setminus K}(X, -)=\Gamma (X,-)\circ \Gamma_{X\setminus K}\rightarrow \Gamma (X,-)\circ \Gamma_{X\setminus f^{-1}(f(K))} =\\= \Gamma (Y, -)\circ f_{*}\circ \Gamma_{X\setminus f^{-1}(f(K))} \cong \Gamma(Y, -)\circ\Gamma_{Y\setminus f(K)}\circ f_{*}=\Gamma_{Y\setminus f(K)}(Y,-)\circ f_{*}.
\end{multline*}

This implies the canonical natural morphism and isomorphism \begin{gather*}
\mathbf{R}\Gamma_{X\setminus K}(X,-)\rightarrow \mathbf{R}\Gamma_{X\setminus f^{-1}(f(K))}(X,-)\cong\mathbf{R}\Gamma_{Y\setminus f(K)}(Y,-)\circ \mathbf{R}f_{*}.
\end{gather*}

Hence we obtain the following commutative diagram for a sheaf $\mathcal{F}\in Sh_{\mathbb{C}}(X)$: 
\[
\begin{diagram} 
\node{\mathbf{R}\Gamma_{K}(X,\mathcal{F})} \arrow{e,t}{} \arrow{s,r}{}
\node{\mathbf{R}\Gamma(X,\mathcal{F})} \arrow{e,t}{} \arrow{s,r}{id}
\node{\mathbf{R}\Gamma_{X\setminus K}(X,\mathcal{F})\to_{+1}} \arrow{s,r}{}
\\
\node{\mathbf{R}\Gamma_{f^{-1}(f(K))}(X,\mathcal{F})} \arrow{e,t}{} \arrow{s,r}{\cong}
\node{\mathbf{R}\Gamma(X,\mathcal{F})} \arrow{e,t}{} \arrow{s,r}{\cong}
\node{\mathbf{R}\Gamma_{X\setminus f^{-1}(f(K))}(X,\mathcal{F})\to_{+1}} \arrow{s,r}{\cong}
\\
\node{\mathbf{R}\Gamma_{f(K)}(Y,\mathbf{R}f_{*}\mathcal{F})} \arrow{e,t}{}
\node{\mathbf{R}\Gamma(Y,\mathbf{R}f_{*}\mathcal{F})} \arrow{e,t}{}
\node{\mathbf{R}\Gamma_{Y\setminus f(K)}(Y,\mathbf{R}f_{*}\mathcal{F})\to_{+1}}
\end{diagram}
\]

It is easy to prove the following lemmas.

\begin{lemma}
\begin{enumerate}
\item Let $f\colon X\to Y$ be a proper surjective continuous map between CLCH-spaces and $\mathcal{F}\in Sh_{\mathbb{C}}(X)$. If $f_{*}\mathcal{F}\in Sh_{\mathbb{C}}(Y)$ is Hartogs and the canonical morphism $\Gamma(X\setminus K,\mathcal{F})\to \Gamma(X\setminus f^{-1}(f(K)),\mathcal{F})$ is a monomorphism for any Hartogs pair $(K, X)$, then $\mathcal{F}$ is Hartogs. 
\item Let $f\colon X\to Y$ be a proper surjective continuous map between CLCH-spaces and $\mathcal{F}\in Sh_{\mathbb{C}}(X)$. If $\mathcal{F}$ is Hartogs, then the canonical map $f_{*}\mathcal{F}$ is Hartogs. 
\end{enumerate}
\end{lemma}

\begin{proof}
\begin{enumerate}
\item The proof follows from the commutative diagram of long exact sequences of the cohomologies corresponding to the commutative diagram of distinguished triangles above.
\item The proof follows from the diagram above. Actually, each compact set $S\subset Y$ has a form $S=f(K)$ for some compact set $K\subset X$; moreover, $(f^{-1}(S),X)$ is a Hartogs pair provided $(S,Y)$ is a Hartogs pair. 
\end{enumerate}
\end{proof}

In particular, for vector bundles we obtain the following proposition.

\begin{proposition}\label{prop1}
Let $f\colon X\to Y$ be a proper surjective holomorphic map between noncompact complex analytic varieties and $\mathcal{F}\in Vect(\mathcal{O}_{X})$. Then $\mathcal{F}$ is Hartogs if and only if $f_{*}\mathcal{F}$ is Hartogs. 
\end{proposition}

\begin{corollary}\label{holoconvex}
Let $X$ be a holomorphically convex noncompact normal complex analytic variety, and let $R\colon X\to Y$ be the corresponding Remmert reduction, $\mathcal{F}\in Vect(\mathcal{O}_{X})$. Then $\mathcal{F}$ is Hartogs if and only if $\dim Y>1$. 
\end{corollary}

\begin{proof}
Recall that $R\colon X\to Y$ is a proper surjective holomorphic map onto the normal Stein variety, and $R_{*}\mathcal{O}_{X}=\mathcal{O}_{Y}$. It follows that $R_{*}\mathcal{F}$ is a locally free $\mathcal{O}_{Y}$-module of finite rank. But $R_{*}\mathcal{F}$ is Hartogs if and only if $\dim Y>1$. 
\end{proof}

\begin{example}\label{exlinearsys}
Let $X'$ be a normal complex projective variety, $D$ be an effective basepoint-free Cartier divisor with connected support, and $X:=X'\setminus \operatorname{Supp}(D)$. The complete linear system associated with $D$ induces a proper holomorphic map $\phi\colon X'\to \mathbb{CP}^{N}$. The Stein factorization implies a proper surjective holomorphic map $\phi\colon X'\to Y'$, where $Y'$ is a normal projective variety such that $\phi_{*}\mathcal{O}_{X'}=\mathcal{O}_{Y'}$, and $D=\phi^{*}(H)$ for an ample divisor $H$ on $Y'$. Note that $Y:=Y'\setminus \operatorname{Supp}(H)$ is a normal Stein variety. It follows that $X$ is holomorphically convex. So,  $\mathcal{O}_{X}$ is Hartogs if and only if $\dim Y>1$. 
\end{example}

Let $f\colon X\to Y$ be a proper surjective continuous map between CLCH-spaces. We have the following natural isomorphism of functors: 
\begin{gather*}
\Gamma_{c}(X,-)\cong \Gamma_{c}(Y,-)\circ f_{*}.
\end{gather*} This implies the following natural isomorphism \begin{gather*}
\mathbf{R}\Gamma_{c}(X,-)\cong\mathbf{R}\Gamma_{c}(Y,-)\circ \mathbf{R}f_{*}.
\end{gather*}

\begin{remark}
Let $f\colon X\to Y$ be a proper surjective holomorphic map between noncompact complex analytic varieties and $\mathcal{F}\in Vect(\mathcal{O}_{X})$. Assume that $\mathcal{F}$ is cohomologically flat at dimension 1 over $Y$ or, equivalently, the function 
\begin{gather*}
Y\to [0,+\infty]\quad y\to \dim_{\mathbb{C}} H^{1}(f^{-1}(y), \mathcal{F}\otimes_{\mathcal{O}_{X}} \mathcal{O}_{f^{-1}(y)})
\end{gather*} is locally constant. Then we obtain the canonical isomorphism $H^{1}_{c}(X, \mathcal{F})\cong H^{1}_{c}(Y,f_{*}\mathcal{F})$. 

It follows from the Leray spectral sequence and the Grauert theorem. Indeed, by the Grauert theorem, it follows that $R^{1}f_{*}\mathcal{F}$ is a locally free $\mathcal{O}_{Y}$-module of finite rank. We get 
\begin{gather*}
E^{0,1}_{2}=\Gamma_{c}(Y,R^{1}f_{*}\mathcal{F})=0,\\
E^{0,1}_{\infty}=0,\\ 
E^{1,0}_{\infty}\cong E^{1,0}_{2}=H^{1}_{c}(Y,f_{*}\mathcal{F}).
\end{gather*}
\end{remark}

\begin{remark}
Let $f\colon X\to Y$ be a proper surjective continuous map between CLCH-spaces. We also have the following natural isomorphism of functors: 
\begin{gather*}
\Gamma (\partial X, -)\cong\Gamma(\partial Y, -)\circ f_{*}.
\end{gather*}

It is easy to see that 
\begin{gather*}
\mathbf{R}\Gamma (\partial X,-)\cong\mathbf{R}\Gamma(\partial Y, -)\circ \mathbf{R}f_{*}.
\end{gather*}

We have the following commutative diagram: 

\[
\begin{diagram} 
\node{\mathbf{R}\Gamma_{c}(X,\mathcal{F})} \arrow{e,t}{} \arrow{s,r}{\cong}
\node{\mathbf{R}\Gamma(X,\mathcal{F})} \arrow{e,t}{} \arrow{s,r}{\cong}
\node{\mathbf{R}\Gamma(\partial X,\mathcal{F})\to_{+1}} \arrow{s,r}{\cong}
\\
\node{\mathbf{R}\Gamma_{c}(Y,\mathbf{R}f_{*}\mathcal{F})} \arrow{e,t}{}
\node{\mathbf{R}\Gamma(Y,\mathbf{R}f_{*}\mathcal{F})} \arrow{e,t}{}
\node{\mathbf{R}\Gamma(\partial Y,\mathbf{R}f_{*}\mathcal{F})\to_{+1}}
\end{diagram}
\]
\end{remark}

Now let $f\colon X\to Y$ be a proper surjective continuous map between CLCH-spaces and $\mathcal{F}\in Sh_{\mathbb{C}}(X)$. Suppose the following conditions hold:
\begin{enumerate}
\item $X$ has only one topological end.
\item There exists a CLCH-space $X'$, a sheaf $\mathcal{F}'\in Sh_{\mathbb{C}}(X')$ with $\sigma_{1}(X',\mathcal{F}')<\infty$, and open embedding $i\colon X\hookrightarrow X'$ such that $\mathcal{F}=i^{-1}\mathcal{F}'$
\end{enumerate}

Since $f$ is a proper map, it follows that $Y$ has only one topological end. Now consider the pushout of maps $f$ and $i$ (in the category of CLCH topological spaces): 

\[
\begin{diagram} 
\node{X} \arrow{e,t}{i} \arrow{s,r}{f}\node{X'} \arrow{s,r}{f'}
\\
\node{Y} \arrow{e,t}{j}\node{Y':=X'\bigsqcup\limits_{X}Y} 
\end{diagram}
\]

\begin{lemma}\label{lemma1}
For the above situation, we have that the map $j$ is an open embedding, $f'$ is a proper surjective continuous map, $Y'$ is a CLCH-space, and the sheaf of $\mathbb{C}$-vector spaces $f'_{*}\mathcal{F}'$ satisfies the following properties:
\begin{enumerate}
\item $\sigma_{1}(Y',f'_{*}\mathcal{F}')\leq \sigma_{1}(X',\mathcal{F}')<\infty$.
\item $j^{-1}f'_{*}\mathcal{F}'=f_{*}\mathcal{F}$.
\end{enumerate}
\end{lemma}
\begin{proof}
Since $i$ is injective, then $j$ is also injective. Since $f$ is proper surjective and $f'|_{X}=f, f'|_{X'\setminus X}=id$, then $f'$ is also proper and surjective. The Leray spectral sequence implies that $\sigma_{1}(Y',f'_{*}\mathcal{F}')\leq \sigma_{1}(X',\mathcal{F}')$. 

Now let $V\subset Y$ be an open subset. The set $j(V)$ is open if and only if $j^{-1}(j(V))=V$ is open and $f'^{-1}(j(V))$ is open. Since $f'^{-1}(j(V))=i(f^{-1}(V))$ and $i$ is open, it follows that $f'^{-1}(j(V))$ is open. 

Further, for any open set $V\subset Y$ we have $j^{-1}f'_{*}\mathcal{F'} (V)= f'_{*}\mathcal{F'} (j(V))=\mathcal{F}'(f'^{-1}j(V))=\mathcal{F}'(i(f^{-1}(V)))=\mathcal{F}(f^{-1}(V))=f_{*}\mathcal{F}(V).$
\end{proof}

The following proposition follows from Proposition \ref{prop1}, Lemma \ref{lemma1}, and Theorem \ref{maintheorem1}.

\begin{proposition}\label{proposition24}
Let $f\colon X\to Y$ be a proper surjective holomorphic map between noncompact complex analytic varieties and $\mathcal{F}\in Vect(\mathcal{O}_{X})$. Suppose $X$ has only one topological end. Assume that there exists a CLCH-space $X'$, a sheaf $\mathcal{F}'\in Sh_{\mathbb{C}}(X')$ with $\sigma_{1}(X',\mathcal{F}')=0$, and an open embedding $i\colon X\hookrightarrow X'$ such that $\mathcal{F}=i^{-1}\mathcal{F}'$. Then the following conditions are equivalent:
\begin{enumerate}
\item $\mathcal{F}$ is Hartogs;
\item $f_{*}\mathcal{F}$ is Hartogs;
\item $H^{1}_{c}(X,\mathcal{F})=0$;
\item $H^{1}_{c}(Y,f_{*}\mathcal{F})=0$. 
\end{enumerate}
\end{proposition}

In the case of the structure sheaf of a noncompact normal complex analytic variety, we obtain the following proposition, which follows from Proposition \ref{prop1}, Lemma \ref{lemma1}, and Theorem \ref{th1'}.

\begin{proposition}\label{proposition25}
Let $f\colon X\to Y$ be a proper surjective holomorphic map between noncompact normal complex analytic varieties, $f_{*}\mathcal{O}_{X}=\mathcal{O}_{Y}$. Suppose $X$ has only one topological end. Assume that there exists a CLCH-space $X'$, a sheaf $\mathcal{F}'\in Sh_{\mathbb{C}}(X')$ with $\sigma_{1}(X',\mathcal{F}')<\infty$, and an open embedding $i\colon X\hookrightarrow X'$ such that $\mathcal{O}_{X}=i^{-1}\mathcal{F}'$. Then the following conditions are equivalent:
\begin{enumerate}
\item $\mathcal{O}_{X}$ is Hartogs; 
\item $\mathcal{O}_{Y}$ is Hartogs; 
\item $\dim_{\mathbb{C}} H^{1}_{c}(X,\mathcal{O}_{X})<\infty$.
\item $\dim_{\mathbb{C}} H^{1}_{c}(Y,\mathcal{O}_{Y})<\infty$. 
\end{enumerate} 
\end{proposition}

\subsection{Reduction to the case of holomorphic line bundles}\label{sectionreduxtoline} 
Now we recall some facts about projective line bundles (see, for instance, \cite[Example 5.17]{Grauert}).

Let $X$ be a complex analytic variety, and let $\mathcal{F}$ be a locally free $\mathcal{O}_{X}$-module of rank $r$. Let $\pi\colon \mathbb{P}(\mathcal{F})\to X$ be the associated projective fiber space. Recall $\mathbb{P}(\mathcal{F})=\operatorname{Proj}(\operatorname{Sym} \mathcal{F})$, where $\operatorname{Sym}\mathcal{F}=\bigoplus\limits_{i\geq 0}\operatorname{Sym}^{i}\mathcal{F}$, and $\pi$ is induced by the canonical morphism $\mathcal{O}_{X}\to \operatorname{Sym}\mathcal{F}$. The variety $\mathbb{P}(\mathcal{F})$ carries a natural line bundle $\mathcal{O}_{\mathbb{P}(\mathcal{F})}(1)$ whose restriction to a fiber $\pi^{-1}(x)\cong \mathbb{P}_{r-1}$ is the line bundle $\mathcal{O}_{\mathbb{P}_{r-1}}(1)$. Moreover, we have the following canonical isomorphisms: 

$\pi_{*}\mathcal{O}_{\mathbb{P}(\mathcal{F})}\cong \mathcal{O}_{X}, \pi_{*}(\mathcal{O}_{\mathbb{P}(\mathcal{F})}(1))\cong \mathcal{F},$ and $R^{q}\pi_{*}(\mathcal{O}_{\mathbb{P}(\mathcal{F})}(1))=0, q>0$. 

The Leray spectral sequence implies the following canonical isomorphisms for any $q\geq 0$: 
$$ H^{q}(\mathbb{P}(\mathcal{F}), \mathcal{O}_{\mathbb{P}(\mathcal{F})}(1))\cong H^{q}(X,\mathcal{F})$$ and since $\pi$ is a proper map, it follows the following canonical isomorphism $$H^{q}_{c}(\mathbb{P}(\mathcal{F}), \mathcal{O}_{\mathbb{P}(\mathcal{F})}(1))\cong H^{q}_{c}(X,\mathcal{F}).$$

Proposition \ref{prop1} implies that the sheaf $\mathcal{O}_{\mathbb{P}(\mathcal{F})}(1)$ is Hartogs if and only if the sheaf $\mathcal{F}$ is Hartogs. 

If $X$ has only one topological end, then $\mathbb{P}(\mathcal{F})$ also has only one topological end. Moreover, we have the following canonical isomorphisms: 
$$\Gamma(\partial\mathbb{P}(\mathcal{F}), \mathcal{O}_{\mathbb{P}(\mathcal{F})})\cong \Gamma(\partial X,\mathcal{O}_{X}),\\\Gamma(\partial\mathbb{P}(\mathcal{F}), \mathcal{O}_{\mathbb{P}(\mathcal{F})}(1))\cong \Gamma(\partial X,\mathcal{F}).$$

Further, we consider the case of holomorphic line bundles over a \textbf{nonsingular} complex analytic variety $X$ (i.e., complex manifold). Denote by $\mathcal{O}(D)$ the invertible $\mathcal{O}_{X}$-module associated with the Cartier divisor $D$ on $X$. Denote by $\mathcal{M}$ the sheaf of meromorphic functions. If $L$ is a holomorphic line bundle over $X$, then we denote by $\mathcal{L}$ the invertible $\mathcal{O}_{X}$-module of holomorphic sections of $L$. 

\begin{lemma}\label{lemmonbund}
Let $X$ be a noncompact complex manifold that has only one topological end. Let $L$ be a holomorphic line bundle over $X$. Assume that $\Gamma(\partial X, \mathcal{O}_{X})\neq \mathbb{C}$ and $\Gamma(\partial X, \mathcal{L})\neq 0$. If $\dim_{\mathbb{C}} H^{1}_{c}(X,\mathcal{L})<\infty$, then there exists a global holomorphic non-zero section $\sigma\in \Gamma(X, \mathcal{L})$. In particular, $\mathcal{L}\cong \mathcal{O}(div(\sigma))$. 
\end{lemma}
\begin{proof}
We have the following long exact sequence: 
\begin{equation*}
	\xymatrix@C=0.5cm{
0 \ar[r] &  \Gamma(X, \mathcal{L}) \ar[rr]^{r} && \Gamma(\partial X, \mathcal{L}) \ar[rr]^{c} && H^{1}_{c}(X,\mathcal{L}) \ar[r] & \cdots }
\end{equation*}

Denote $m=\dim_{\mathbb{C}} H^{1}_{c}(X,\mathcal{L})$. Let $s\in \Gamma(\partial X, \mathcal{L})$ be a non-zero holomorphic section, and assume that $c(s)\neq 0$. Let $f\in \Gamma(\partial X, \mathcal{O}_{X})$ be a nonconstant holomorphic function. 

Consider the following holomorphic sections: $s,fs,\cdots, f^{m}s\in \Gamma(\partial X, \mathcal{L})$. We may assume that $c(f^{i}s)\neq 0$ for any $i=1,\cdots m$. Then there exists a non-zero polynomial $P\in\mathbb{C}[T]$ such that $c(P(f)s)=0$. So, there exists a global holomorphic non-zero section $\sigma\in \Gamma(X, \mathcal{L})$ such that $r(\sigma)=P(f)s$.
\end{proof}

For a Cartier divisor $D=\sum_{j\in I} a_{j}D_{j}$ ($a_{j}\neq 0$), we define the Cartier divisor $D_{red}:=\sum_{j\in I} D_{j}$, and for any $b=(b_{j})_{j\in I}$ such that $b_{j}\in \{0,1\}$ we define $D_{b}:=\sum_{j\in I} b_{j}D_{j}.$

\begin{lemma}\label{lemmondiv}
Let $X$ be a noncompact complex manifold that has only one topological end, and let $D$ be an effective Cartier divisor. For a given $b$, the following conditions hold:
\begin{enumerate}
\item $\dim_{\mathbb{C}} H^{1}_{c}(X,\mathcal{O}(D))<\infty$ implies $\dim_{\mathbb{C}} H^{1}_{c}(X,\mathcal{O}(D-D_{b}))<\infty$
\item $\dim_{\mathbb{C}} H^{1}_{c}(X,\mathcal{O}(D+D_{b}))<\infty$ implies $\dim_{\mathbb{C}} H^{1}_{c}(X,\mathcal{O}(D))<\infty$
\end{enumerate}

In particular, if the divisor $D_{b}$ has a noncompact support, then the following canonical homomorphisms are monomorphisms:  
\begin{enumerate}
\item $H^{1}_{c}(X,\mathcal{O}(D-D_{b}))\to H^{1}_{c}(X,\mathcal{O}(D))$
\item $H^{1}_{c}(X,\mathcal{O}(D))\to H^{1}_{c}(X,\mathcal{O}(D+D_{b}))$
\end{enumerate}
\end{lemma}
\begin{proof}
Consider the following canonical exact sequence: 
\begin{gather*}
0\to \mathcal{O}(-D_b)\to \mathcal{O}_{X}\to \mathcal{O}_{D_b}\to 0,
\end{gather*}
where $\mathcal{O}_{D_b}:=\mathcal{O}_{X}/\mathcal{O}(-D_b)$. 

Tensoring this sequence by $\mathcal{O}(D)$, we obtain the following exact sequence:  
\begin{gather*}
0\to \mathcal{O}(D-D_b)\to \mathcal{O}(D)\to \mathcal{O}_{D_b}(D)\to 0,
\end{gather*}
where $\mathcal{O}_{D_b}(D)=(\mathcal{O}_{X}/\mathcal{O}(-D_b))\otimes_{\mathcal{O}_{X}}\mathcal{O}(D)$. 

Let us note that $(D_b,\mathcal{O}_{D_b})$ is a reduced complex analytic space. If $(D_b,\mathcal{O}_{D_b})$ is a compact analytic space, then $\dim_{\mathbb{C}} \Gamma_{c}(X,\mathcal{O}_{D_b}(D))<\infty$. Hence $\dim_{\mathbb{C}} H^{1}_{c}(X,\mathcal{O}(D))<\infty$ implies $\dim_{\mathbb{C}} H^{1}_{c}(X,\mathcal{O}(D-D_{b}))<\infty$. If $(D_b,\mathcal{O}_{D_b})$ is a noncompact analytic space, then $\Gamma_{c}(X,\mathcal{O}_{D_b}(D))=0$. Hence the canonical homomorphism 
$$H^{1}_{c}(X,\mathcal{O}(D-D_b)\to H^{1}_{c}(X,\mathcal{O}(D))$$ is a monomorphism. 

For the second statement, we need to consider the following canonical exact sequence: 
\begin{gather*} 
0\to \mathcal{O}_{X}\to \mathcal{O}([D]_b)\to \mathcal{O}_{[D]_b}([D]_b)\to 0.
\end{gather*}
\end{proof}

\begin{corollary}\label{corbundltoholom}
Let $X$ be a noncompact complex manifold that has only one topological end, and let $L$ be a holomorphic line bundle over $X$. Assume that $\Gamma(\partial X, \mathcal{O}_{X})\neq \mathbb{C}$ and $\Gamma(\partial X, \mathcal{L})\neq 0$. If $\dim_{\mathbb{C}} H^{1}_{c}(X,\mathcal{L})<\infty$, then $\dim_{\mathbb{C}} H^{1}_{c}(X,\mathcal{O}_{X})<\infty$. In particular, $\mathcal{O}_{X}$ is Hartogs.
\end{corollary}
\begin{proof}
Lemma \ref{lemmonbund} implies that $\mathcal{L}=\mathcal{O}(D)$ for an effective Cartier divisor $D=\sum a_{i}D_{i}$, $a_{i}\in \mathbb{Z}_{\geq 1}$. Note that $D-D_{red}=\sum\limits_{i: a_{i}\geq 2}(a_{i}-1)D_{i}$. Lemma \ref{lemmondiv} implies that $\dim_{\mathbb{C}} H^{1}_{c}(X,\mathcal{O}(D-D_{red})<\infty.$

Further, for the divisor $D'=D-D_{red}$, we may repeat the same arguments. Since the sum $D=\sum a_{i}D_{i}$ is locally finite, it follows that $\dim_{\mathbb{C}} H^{1}_{c}(X,\mathcal{O}_{X})<\infty$. The last statement follows from Proposition \ref{lemmonholom}.
\end{proof}

\begin{remark}\label{corbundltoholom1}
Let $X$ be a noncompact complex manifold that has only one topological end, and let $L$ be a holomorphic line bundle over $X$. So, the conditions $\Gamma(\partial X, \mathcal{O}_{X})\neq \mathbb{C}$, $\Gamma(\partial X, \mathcal{L})\neq 0$, and $\dim_{\mathbb{C}} H^{1}_{c}(X,\mathcal{L})<\infty$ implies that $\Gamma(X,\mathcal{L})\neq 0$, $\Gamma(X,\mathcal{O}_{X})\neq\mathbb{C}$.
\end{remark}

\begin{definition}\label{defi22}
Let $\mathcal{F}$ be a coherent $\mathcal{O}_{X}$-module over a complex analytic variety $X$. An irreducible divisor $D$ of $X$ is called \textbf{$\mathcal{F}$-removable} if the restriction homomorphism 
$$\Gamma(X,\mathcal{F})\to \Gamma(X\setminus D,\mathcal{F})$$ is an epimorphism. 
\end{definition}

\begin{example}\label{removexample}
Each exceptional irreducible divisor is $\mathcal{F}$-removable for any locally free $\mathcal{O}_{X}$-module of finite rank $\mathcal{F}$. Let us remember that an irreducible divisor $D$ of a complex analytic variety $X$ is called \textbf{exceptional} if there exists a complex analytic variety $Y$ and a proper surjective holomorphic map $\phi\colon X\to Y$ such that $ \operatorname{codim} (\phi(D))\geq 2$ and $\phi\colon X\setminus D\to Y\setminus \phi(D)$ is a biholomorphism. 
\end{example}

\begin{proposition}\label{mainproponline}
Let $X$ be a noncompact complex manifold which has only one topological end, and let $L$ be a holomorphic line bundle over $X$. Assume that each compact irreducible Cartier divisor of $X$ is $\mathcal{L}$-removable, $\Gamma(\partial X, \mathcal{O}_{X})\neq \mathbb{C}$, and $\Gamma(\partial X, \mathcal{L})\neq 0$. If $\dim_{\mathbb{C}} H^{1}_{c}(X,\mathcal{L})<\infty,$ then the canonical homomorphism $\Gamma(X,\mathcal{L})\to \Gamma(\partial X,\mathcal{L})$ is an isomorphism. In particular, the sheaf $\mathcal{L}$ is Hartogs. 
\end{proposition}
\begin{proof}
We have the following long exact sequence: 
\begin{equation*}
	\xymatrix@C=0.5cm{
0 \ar[r] &  \Gamma(X, \mathcal{L}) \ar[rr]^{r} && \Gamma(\partial X, \mathcal{L}) \ar[rr]^{c} && H^{1}_{c}(X,\mathcal{L}) \ar[r] & \cdots }
\end{equation*}

Lemma \ref{lemmonbund} implies that there exists an effective Cartier divisor $D$ on $X$ such that $\mathcal{L}=\mathcal{O}(D)$. Remark \ref{corbundltoholom1} implies that $\Gamma(X,\mathcal{O}_{X})\neq \mathbb{C}$. 

As in the proof of Lemma \ref{lemmonbund}, for any $s\in \Gamma(\partial X, \mathcal{O}(D))$ (which is represented by a section $\sigma\in\Gamma(X\setminus K,\mathcal{O}(D))$ for a sufficiently large compact set $K\subset X$) there exists a global holomorphic function $g\in \Gamma(X,\mathcal{O}_{X})$ and there exists a global holomorphic section $\sigma'\in \Gamma(X,\mathcal{O}(D))$ such that $r(\sigma')=gs$. 

Consider the global meromorphic section $s'=\frac{\sigma'}{g}\in \Gamma(X,\mathcal{M}\otimes\mathcal{O}(D))$. Note that $$div(s')|_{X\setminus K}=div(\sigma)\geq 0$$ for a sufficiently large compact set $K\subset X$. Consider $div(s')=\sum a_{i}D_i$, where $D_{i}$ are irreducible Cartier divisors. If $a_{i}<0$, then $D_{i}$ is a compact set. It follows that the meromorphic section $s'$ has only compact poles. Since each compact irreducible Cartier divisor of $X$ is $\mathcal{O}(D)$-removable, it follows that $s'\in \Gamma(X,\mathcal{O}(D))$. The last statement follows from Lemma \ref{mainlemma-1}.
\end{proof}

\begin{theorem}\label{mainthsheaf}
Let $X$ be a noncompact complex manifold that has only one topological end, $\mathcal{F}\in Vect(\mathcal{O}_{X})$. Assume that each compact irreducible Cartier divisor of $X$ is $\mathcal{F}$-removable, $\Gamma(\partial X, \mathcal{O}_{X})\neq \mathbb{C}$, and $\Gamma(\partial X, \mathcal{F})\neq 0$. If $\dim_{\mathbb{C}} H^{1}_{c}(X,\mathcal{F})<\infty$, then the canonical homomorphism 
\begin{gather*}
\Gamma(X,\mathcal{F})\to \Gamma(\partial X,\mathcal{F})
\end{gather*}
 is an isomorphism. In particular, the sheaf $\mathcal{F}$ is Hartogs.
\end{theorem}

\begin{proof}
Consider the projective fiber bundle $\pi\colon\mathbb{P}(\mathcal{F})\to X$. The following conditions hold: 

\begin{enumerate}
\item $\mathbb{P}(\mathcal{F})$ has only one topological end. 
\item $\Gamma(\partial\mathbb{P}(\mathcal{F}), \mathcal{O}_{\mathbb{P}(\mathcal{F})})\neq \mathbb{C}$.
\item $\Gamma(\partial\mathbb{P}(\mathcal{F}), \mathcal{O}_{\mathbb{P}(\mathcal{F})}(1))\neq 0$.
\item $\dim_{\mathbb{C}} H^{1}_{c}(\mathbb{P}(\mathcal{F}), \mathcal{O}_{\mathbb{P}(\mathcal{F})}(1))<\infty$.
\end{enumerate}

If $D$ is a compact irreducible Cartier divisor of $\mathbb{P}(\mathcal{F})$, then $\pi(D)$ is a compact subset. Since $X$ is noncompact, then $\pi(D)\neq X$. The Remmert's proper mapping theorem implies that $\pi(D)$ is an analytic subset of $X$. We have the following commutative diagram of the restriction homomorphisms:

\[\begin{diagram} 
\node{\Gamma(X,\mathcal{F})} \arrow{e,t}{r} \arrow{s,r}{\cong} 
\node{\Gamma(X\setminus \pi(D), \mathcal{F})} \arrow{s,r}{\cong}
\\
\node{\Gamma(\mathbb{P}(\mathcal{F}),\mathcal{O}_{\mathbb{P}(\mathcal{F}}(1))} \arrow{e,t}{r_{1}} \arrow{se,r}{r_{2}} 
\node{\Gamma(\mathbb{P}(\mathcal{F})\setminus \pi^{-1}(\pi(D)),\mathcal{O}_{\mathbb{P}(\mathcal{F}}(1))} 
\\
\node[2]{\Gamma(\mathbb{P}(\mathcal{F})\setminus D, \mathcal{O}_{\mathbb{P}(\mathcal{F}}(1))}\arrow{n,r}{r_3}\\
\end{diagram}\]

If $ \operatorname{codim} (\pi(D))>1$, then $r$ is an isomorphism (by the Serre theorem \cite[Chapter II, Theorem 5.29]{Grauert}). If $ \operatorname{codim} (\pi(D))=1$, then $r$ is an isomorphism by assumption. In both cases, we obtain that $r_{1}$ is an isomorphism. 

Since $r_{2}, r_{3}$ are monomorphisms, it follows that $r_2$ is an isomorphism. Hence each compact irreducible Cartier divisor of $\mathbb{P}(\mathcal{F})$ is $\mathcal{O}_{\mathbb{P}(\mathcal{F}}(1)$-removable. Proposition \ref{mainproponline} implies that the restriction homomorphism $$\Gamma(\mathbb{P}(\mathcal{F}), \mathcal{O}_{\mathbb{P}(\mathcal{F})}(1))\to \Gamma(\partial\mathbb{P}(\mathcal{F}), \mathcal{O}_{\mathbb{P}(\mathcal{F})}(1))$$ is an isomorphism. In particular, $\mathcal{O}_{\mathbb{P}(\mathcal{F})}(1)$ is Hartogs. This concludes the proof.
\end{proof}

\begin{example}\noindent
Since each compact irreducible Cartier divisor of a 1-convex complex manifold is exceptional, it follows that it is $\mathcal{F}$-removable for any locally free $\mathcal{O}_{X}$-module of finite rank $\mathcal{F}$. By \cite[Proposition 1]{Viorel2}, it follows that $\dim_{\mathbb{C}} H^{1}_{c}(X,\mathcal{F})<\infty$ for any locally free $\mathcal{O}_{X}$-module of finite rank. Hence $\mathcal{F}$ is Hartogs.  
\end{example}

\newpage
\section{$(b,\sigma)$-Compactified pairs}\label{sec3}

\subsection{Fibered categories}
In this section, we consider categories that are obtained via the following Grothendieck construction (see, for instance, \cite{Vistoli, Grothendieck}): let $B$ be a small category, $\mathbf{Cat}$ be the category of all (small) categories, and $T\colon B\to \mathbf{Cat}^{op}$ be a (contravariant)  pseudo-functor. We form the following category, which is denoted by $B_{T}$. 
\begin{enumerate}
\item Objects of $B_{T}$: $Ob(B_T):=\{(x,y)\mid x\in Ob(B), y\in Ob(Tx)\}$; 
\item Morphisms of $B_T$: for any $(x,y),(x',y')\in Ob(B_T)$ define 
\begin{gather*}
Hom_{B_T}((x,y),(x',y')):=\{(f,g)\mid f\in Hom_{B}(x,x'), g\in Hom_{Tx}((Tf)(y'),y)\}
\end{gather*}
\item Composition of morphisms: 
\begin{gather*}
\circ\colon Hom_{B_T}((x,y),(x',y'))\times Hom_{B_T}((x',y'),(x'',y''))\to Hom_{B_T}((x,y),(x'',y'')),\\ (f',g')\circ (f,g):= (f'\circ f, g\circ (Tf)(g')).
\end{gather*}
\item Identity: for any $(x,y)\in Ob(B_T)$ define $1_{(x,y)}=(1_{x},1_{y})$, where $1_{x}$ is the identity morphism of $x\in Ob(B)$, $1_{y}$ is the identity morphism of $y \in Ob(Tx)$. 
\end{enumerate}

It is easy to show that we indeed obtain a category. Note that the morphism 
\begin{gather*}
(f,g)\colon (x,y)\to (x',y')
\end{gather*}
 is an isomorphism if and only if $f\colon x \to x'$ is an isomorphism and $g\colon (Tf)(y')\to y$ is an isomorphism such that $g\circ (Tf)(g^{-1})=1_{y}, g^{-1}\circ (Tf^{-1})(g)=1_{y'}.$ 
 
\begin{example}
The category $Top_{Sh}$, where $Top$ is the category of topological spaces and $Sh$ is the following pseudo-functor (indeed, it is a functor): each $X\in Top$ corresponds to the category $Sh(X)$ (i.e., the category of sheaves of abelian groups over $X$), and each continuous map $f\colon X\to Y$ corresponds to the inverse image functor $Sh(f):=f^{-1}\colon Sh(Y)\to Sh(X)$. 
\end{example}

\subsection{Fibered categories which are considered and $(b,\sigma)$-compactified pairs}

\begin{definition}\label{remoncategory}
We list below the pairs $(B,T)$ such that the corresponding categories $B_{T}$ are considered.
\begin{enumerate}  
\item Categories: complex analytic varieties $An$; normal complex analytic varieties $An_{norm}$; complex analytic manifolds $An_{sm}$. Pseudo-functors:
\begin{enumerate}
\item $Sh_{\mathbb{C}}$: each object $(X,\mathcal{O}_{X})$ corresponds to the category $Sh_{\mathbb{C}}(X)$ of sheaves of $\mathbb{C}$-vector spaces over $X$; each morphism $(f,f^{\#})\colon (X,\mathcal{O}_{X})\to (Y,\mathcal{O}_{Y})$ (here $f\colon X\to Y$ is a continuous map, $f^{\#}\colon f^{-1}\mathcal{O}_{Y}\to \mathcal{O}_{X}$ is a morphism of sheaves) corresponds to the inverse image functor $$f^{-1}\colon Sh_{\mathbb{C}}(Y)\to Sh_{\mathbb{C}}(X).$$ 
\item $Id$: each object $(X,\mathcal{O}_{X})$ corresponds to the structure sheaf $\mathcal{O}_{X}$ considered as the category with one object and one morphism (identity); each morphism $(f,f^{\#})\colon (X,\mathcal{O}_{X})\to (Y,\mathcal{O}_{Y})$ corresponds to the functor $Id(f,f^{\#})\colon \mathcal{O}_{Y}\to \mathcal{O}_{X}$ defined via the canonical isomorphism $$f^{\#}\otimes_{f^{-1}\mathcal{O}_{Y}}\mathcal{O}_{X}\colon f^{-1}\mathcal{O}_{Y}\otimes_{f^{-1}\mathcal{O}_{Y}}\mathcal{O}_{X}\cong \mathcal{O}_{X}.$$  
\item $Coh$: each object $(X,\mathcal{O}_{X})$ corresponds to the category $Coh(\mathcal{O}_{X})$ of coherent $\mathcal{O}_{X}$-modules; each morphism $(f,g)\colon (X,\mathcal{O}_{X})\to (Y,\mathcal{O}_{Y})$ corresponds to the analytic inverse image functor 
\begin{gather*}
f^{*}\colon Coh(\mathcal{O}_{Y})\to Coh(\mathcal{O}_{X}), \mathcal{F}\mapsto f^{*}\mathcal{F}=f^{-1}\mathcal{F}\otimes_{f^{-1}\mathcal{O}_{Y}}\mathcal{O}_{X}.
\end{gather*}
\item $Vect$: each object $(X,\mathcal{O}_{X})$ corresponds to the category $Vect(\mathcal{O}_{X})$ of locally free $\mathcal{O}_{X}$-modules of finite ranks; each morphism $(f,g)\colon (X,\mathcal{O}_{X})\to (Y,\mathcal{O}_{Y})$ corresponds to the analytic inverse image functor $$f^{*}\colon Vect(\mathcal{O}_{Y})\to Vect(\mathcal{O}_{X}).$$ 
\end{enumerate}
\item Categories: $Al^{a}, Al_{norm}^{a}, Al_{sm}^{a}$ which are analytifications of the categories of complex algebraic varieties $Al$, of normal complex algebraic varieties $Al_{norm}$, of nonsingular complex algebraic varieties $Al_{sm}$, respectively; Pseudo-functors: $Sh_{\mathbb{C}}, Id, Coh, Vect$. 
\item Category of complex analytic $G$-varieties $An_{G}$ (here $G$ is a complex Lie group); Pseudo-functor: $Id$.
\item Categories: $Al_{G}^{a}, Al_{G,norm}^{a}, Al_{G, sm}^{a}$ (here $G$ is a complex algebraic group) which are analytifications of the categories of complex algebraic $G$-varieties, of normal complex algebraic $G$-varieties, of complex algebraic $G$-manifolds, respectively; Pseudo-functor: $Id$. 
\end{enumerate} 
\end{definition}

\begin{remark}
Let $B$ be as in Definition \ref{remoncategory}. The category $B_{Id}$ is equivalent to $B$, and we denote $B_{Id}$ by $B$.
\end{remark}

\begin{definition}\label{Defibsigma}
Fix a category $B_T$ as in Definition \ref{remoncategory}. Let $(X,\mathcal{F})$ be an object of $B_T$ such that $X$ is noncompact, and let $b,\sigma\in [0,\infty]$.
\begin{itemize}
\item The pair $(X,\mathcal{F})$ is called \textbf{$(b,\sigma)$-compactified} by the morphism $$(i,k)\colon (X, \mathcal{F})\to (X',\mathcal{F}')$$ in the category $B_T$ if the following conditions holds:
\begin{enumerate}
\item $X'$ is compact.
\item $i\colon X\hookrightarrow X'$ is an open immersion (i.e., topological open immersion with $i^{-1}\mathcal{O}_{X'}\cong \mathcal{O}_{X}$).
\item $k\colon (Ti)(\mathcal{F}')\to \mathcal{F}$ is an isomorphism.
\item $X'\setminus i(X)$ is a proper analytic set that has only $b$ connected components.
\item $\dim_{\mathbb{C}}H^{1}(X',\mathcal{F}')=\sigma$.
\end{enumerate}
\item If $T=Id$, then we say that $X$ is \textbf{$(b,\sigma)$-compactified} by the morphism $i\colon X\hookrightarrow X'$. 
\item We say that $X$ is \textbf{$(b,\sigma)$-compactifiable} if $X$ is $(b,\sigma)$-compactified by some morphism $i\colon X\hookrightarrow X'$.
\end{itemize}
\end{definition}

\begin{remark}
The notion of compactifiable complex manifolds may have first arisen in Kawamata's paper \cite{Kaw} in the deformation theory of non-compact manifolds context. See also \cite{BALLICO} for more about this context. The questions are about a classification of compactifiable manifolds considered in Enoki's paper \cite{Enoki} and, for instance, Vo Van Tan's papers \cite{Tan1, Tan2, Tan3}.
\end{remark}

\begin{example}
Let $X=\mathbb{C}^{*}$ and $\mathcal{F}=\mathcal{O}_{\mathbb{C}^{*}}$ be the sheaf of holomorphic functions. In the category $An_{Sh_{\mathbb{C}}}$, we can compactify $X$ in the following ways.
\begin{enumerate}
\item The pair $(\mathbb{C}^{*}, \mathcal{O}_{\mathbb{C}^{*}})$ is $(2,\infty)$-compactified by the morphism 
\begin{gather*}
(i,k)\colon (\mathbb{C}^{*}, \mathcal{O}_{\mathbb{C}^{*}})\to (\mathbb{CP}^{1},i_{!}\mathcal{O}_{\mathbb{C}^{*}}),
\end{gather*} where $i\colon \mathbb{C}^{*}\hookrightarrow \mathbb{CP}^{1}$ is the canonical open immersion, $k\colon i^{-1}i_{!}\mathcal{O}_{\mathbb{C}^{*}}\cong \mathcal{O}_{\mathbb{C}^{*}}$ is the canonical isomorphism. Indeed, we have the following short exact sequence 
\begin{gather*}
0\to \mathbb{C}\to \mathcal{O}_{\mathbb{CP}^{1},0}\oplus \mathcal{O}_{\mathbb{CP}^{1},\infty}\to H^{1}_{c}(\mathbb{C}^{*},\mathcal{O}_{\mathbb{C}^{*}})=H^{1}(\mathbb{CP}^{1}, i_{!}\mathcal{O}_{\mathbb{C}^{*}})\to 0.
\end{gather*}
It follows that $\sigma=\infty$.
\item The pair $(\mathbb{C}^{*}, \mathcal{O}_{\mathbb{C}^{*}})$ is $(2,\sigma)$-compactified by the morphism 
\begin{gather*}
(i,k)\colon (\mathbb{C}^{*}, \mathcal{O}_{\mathbb{C}^{*}})\to (\mathbb{CP}^{1},\mathcal{O}_{\mathbb{CP}^{1}}(n)),
\end{gather*} 
where $i\colon \mathbb{C}^{*}\hookrightarrow \mathbb{CP}^{1}$ is the canonical open immersion, $k\colon i^{-1}\mathcal{O}_{\mathbb{CP}^{1}}(n)\cong \mathcal{O}_{\mathbb{C}^{*}}$ is a trivialization isomorphism, and 
\begin{gather*}
\sigma=\begin{cases}
0, & \text{if $n\geq -1$;} \\
-n-1, & \text{if $n<-1$.}
\end{cases}
\end{gather*}

\item The pair $(\mathbb{C}^{*}, \mathcal{O}_{X})$ is $(1,1)$-compactified by the morphism 
\begin{gather*}
(i,k)\colon (\mathbb{C}^{*}, \mathcal{O}_{\mathbb{C}^{*}})\to (X',\mathcal{O}_{X'}),
\end{gather*} 
where $X' \subset \mathbb{CP}^{2}$ is defined by the equation $z_{1}^{3}+z_{2}^{3}-z_0 z_1 z_2=0$, $i\colon \mathbb{C}^{*}\hookrightarrow X'$ is the canonical open immersion, and $k\colon i^{-1}\mathcal{O}_{X'}\cong \mathcal{O}_{\mathbb{C}^{*}}$ is the canonical isomorphism. 

Recall that the sheaf of ideals of $X'$ and the canonical sheaf of $\mathbb{CP}^{2}$ are isomorphic to $\mathcal{O}_{\mathbb{CP}^{2}}(-3H)$. The short exact sequence $0\to \mathcal{I}\to\mathcal{O}_{\mathbb{CP}^{2}}\to \mathcal{O}_{\mathbb{CP}^{2}}/\mathcal{I}\to 0$, the Serre duality, and $H^{i}(\mathbb{CP}^{2},\mathcal{O}_{\mathbb{CP}^{2}})=0,\forall i>0$ implies that 
\begin{multline*}
\sigma=\dim_{\mathbb{C}}(H^{1}(\mathbb{CP}^{2},\mathcal{O}_{\mathbb{CP}^{2}}/\mathcal{I}))=\dim_{\mathbb{C}}(H^{2}(\mathbb{CP}^{2},\mathcal{I}))=\\=\dim_{\mathbb{C}}(H^{0}(\mathbb{CP}^{2}, \mathcal{I}^{*}\otimes K_{\mathbb{CP}^{2}}))=\dim_{\mathbb{C}}(H^{0}(\mathbb{CP}^{2}, \mathcal{O}_{\mathbb{CP}^{2}})=1
\end{multline*}
\end{enumerate}
\end{example}

\begin{example}
Fix a category $B_T$ as in Definition \ref{remoncategory}. If $(X,\mathcal{F})$ is $(b,\sigma)$-compactified by a morphism $(X,\mathcal{F})\to (X',\mathcal{F}')$, then for any $k\in\mathbb{Z}$ the pair $(X,\mathcal{F})$ is $(b,\sigma')$-compactified by a morphism $(X,\mathcal{F})\to (X,\mathcal{F}'\otimes\mathcal{I}_{Z}^{\otimes k})$, where $\mathcal{I}_{Z}$ is the sheaf of ideals of $Z$ and $\sigma'=\dim_{\mathbb{C}} H^{1}(X',\mathcal{F}'\otimes\mathcal{I}_{Z}^{\otimes k})$.

Assume that $B$ is $An_{Coh}$ or $An_{Vect}$. Assume that $Z:=X'\setminus X$ is connected and the support of a divisor $D$. Let $\mathcal{O}_{X'}(D)$ be the corresponding line bundle, which is positive. The Grauert vanishing theorem (\cite[Chapter VI, Theorem 4.3]{Grauert}) implies that if $(X,\mathcal{F})$ is $(1,\sigma)$-compactified by a morphism $(X,\mathcal{F})\to (X',\mathcal{F}')$, then $(X,\mathcal{F})$ is $(1,0)$-compactified by a morphism $(X,\mathcal{F})\to (X',\mathcal{F}'\otimes\mathcal{O}_{X'}(kD))$ for some $k\in\mathbb{Z}_{\geq 0}$.

Now, fix the category $(An_{sm})_{Vect}$ and let $Z:=X'\setminus X$ be connected and the support of a divisor $D$. If $(X,\mathcal{F})$ is $(1,\sigma)$-compactified by a morphism $(X,\mathcal{F})\to (X',\mathcal{F}')$, then $(X,\mathcal{F})$ is $(1,\sigma(k))$-compactified by a morphism $(X,\mathcal{F})\to (X',\mathcal{F}'\otimes\mathcal{O}_{X'}(kD))$ where $\sigma(k)$ is a function which has some asymptotic inequalities as $k\to +\infty$ (see the weak Morse inequalities \cite[Section 8]{Dem}). In particular, if $D$ is nef and $X'$ is K\"ahler, then $\sigma(k)=o(k^{\dim X'})$. 
\end{example}

\subsection{On $(b,\sigma)$-compactifiable varieties}\label{sectoncompact}
In this subsection, we give any remarks on numbers $(b,\sigma)$ for the categories $B$ as in Definition \ref{remoncategory}.

Note that the number $b$ is related to the topological ends of $X$ (see \cite{Freudenthal, Peschke} or Remark \ref{topologicalend} for more details about topological ends). Suppose $X$ has $e(X)$ topological ends. In general, $b\leq e(X)$, but in the category $An_{norm}$ we have equality $b=e(X)$; in fact, the Riemann extension theorem implies that if $X$ is a connected normal space, then for each thin set $A$ in $X$ the space $X\setminus A$ is connected (\cite[Chapter 1, \S 13]{Grauert}). It follows that the number $b$ does not depend on compactifications of $X$ in $An_{norm}$. 

In the category $An_{sm}$, the number $\sigma$ is a bimeromorphic invariant. Namely, if $X$ is $(b,\sigma')$-compactified by $X\hookrightarrow X'$, $(b,\sigma'')$-compactified by $X\hookrightarrow X''$ in the category $An_{sm}$, and $X'$ and $X''$ are bimeromorphically equivalent, then $\sigma'=\sigma''$ (see, for instance, \cite[Corollary 1.4]{RaoSong}). 

The Nagata theorem \cite{Nagata} implies that each normal complex algebraic variety admits a compactification in the category $Al_{norm}^{a}$. By the resolution of singularities theorem, it follows that each complex algebraic manifold admits a compactification in the category $Al_{sm}^{a}$.

In the category $Al_{sm}^{a}$, the number $\sigma$ coincides with the dimension of the Albanese variety (see \cite{SerreAlbanes} for more details about Albanese varieties). Recall that if $X'$ is a compact complex algebraic manifold and $ \operatorname{Alb}(X')$ is the Albanese variety of $X'$, then 
\begin{gather*}
\dim_\mathbb{C}  \operatorname{Alb}(X')=\dim_\mathbb{C} \Gamma (X',\Omega^{1}_{X'})=\dim_\mathbb{C} H^{1}(X',\mathcal{O}_{X'}).
\end{gather*}

Now if $X$ is $(b,\sigma)$-compactified by $X\hookrightarrow X'$ in $Al_{sm}^{a}$, then $\sigma=\dim_{\mathbb{C}}  \operatorname{Alb}(X')$. Moreover, since the Albanese variety is a birational invariant, it follows that $\sigma=\dim_{\mathbb{C}}  \operatorname{Alb}(X)$.

By the remarks above, it follows that the numbers $(b,\sigma)$ do not depend on any compactification in the category $Al_{sm}^{a}$. Namely, we obtain the following proposition:

\begin{proposition}\label{prop27}
A complex algebraic manifold is $(b,\sigma)$-compactifiable in $Al_{sm}^{a}$ if and only if $X$ has $b$ topological ends and $\dim_{\mathbb{C}}  \operatorname{Alb}(X)=\sigma$.
\end{proposition}

\begin{remark}\label{RemonSumi}
By the Sumihiro theorem \cite[Theorem 3]{Sumihiro}, it follows that if $G$ is a connected complex linear algebraic group, then each noncompact normal complex algebraic $G$-variety admits a compactification in the category $Al_{G,norm}^{a}$ (i.e., a $G$-equivariant compactification). By Brion's results \cite{Brion2}, it follows that if $G$ is an arbitrary connected complex algebraic group, then each noncompact normal quasiprojective complex algebraic $G$-variety admits a compactification in the category $Al_{G,norm}^{a}$.
\end{remark}
\subsection{Almost homogeneous algebraic G-manifolds}\label{sectionalmost}
A complex analytic $G$-variety $X$ is called \textbf{almost homogeneous} if $X$ has an open $G$-orbit (which we denote by $\Omega$). Note that an open $G$-orbit $\Omega$ is unique and connected, and $E:=X\setminus \Omega$ is a proper analytic subset \cite[Section 1.7, Proposition 4]{Akhiezer}.

In this section, we only consider analytifications of almost homogeneous algebraic $G$-manifolds with respect to an algebraic action of a connected complex algebraic group $G$.

\begin{remark}
We list some properties of Albanese varieties of algebraic groups and almost homogeneous algebraic varieties (see \cite{BrionLoghomo, Brionnonaffine, Brion2, Brion3}). Let $G$ be a connected complex algebraic group, $G/H$ be an algebraic homogeneous $G$-manifold, and $X$ be an almost homogeneous $G$-manifold with open $G$-orbit $\Omega\cong G/H$.
\begin{enumerate}
\item There exists a minimal closed normal algebraic subgroup $G_{aff}\subset G$ such that the quotient $G/G_{aff}$ is an abelian variety. Note that $G_{aff}$ is an affine connected subgroup. 
\item The Albanese variety of $G$ is $ \operatorname{Alb}(G)=G/G_{aff}$, and the Albanese morphism is the quotient holomorphic map $G\to G/G_{aff}$.
\item The Albanese variety of $G/H$ is $ \operatorname{Alb}(G/H)=G/G_{aff}H$, and the Albanese morphism is the quotient holomorphic map $G/H\to G/G_{aff}H$.
\item The canonical map $ \operatorname{Alb}(\Omega)\to  \operatorname{Alb}(X)$ is an isomorphism, the Albanese variety $ \operatorname{Alb}(X)$ is a homogeneous algebraic $G$-manifold, and the Albanese morphism $X\to  \operatorname{Alb}(X)$ is a surjective $G$-equivariant holomorphic map. 
\item Let $I:=G_{aff}H$. The Albanese map $\alpha\colon X\to  \operatorname{Alb}(\Omega)=G/I$ is a holomorphic fiber bundle with fiber $Y:=\alpha^{-1}(eI)$ (which is a complex $I$-manifold). 
\item Let $G\times^{I}Y$ be the fiber bundle associated with the principal $I$-bundle $G\to G/I$ and complex $I$-manifold $Y$. We have that $X\cong G\times^{I}Y$ as fiber bundles. Moreover, $Y$ is an almost homogeneous $G_{aff}$-manifold with open $G_{aff}$-orbit $G_{aff}H/H\cong G_{aff}/G_{aff}\cap H$.
\end{enumerate}
\end{remark}

\begin{remark}\label{remar12}
If an almost homogeneous complex algebraic $G$-manifold $X$ with open $G$-orbit $\Omega$ is $(b,\sigma)$-compactifiable in $Al_{G, sm}^{a}$, then $X$ has $b$ topological ends, $\sigma = \dim_{\mathbb{C}}  \operatorname{Alb}(\Omega)$. In particular, if $G$ is a connected complex \textbf{linear} algebraic group, then $ \operatorname{Alb}(\Omega)$ is a point. It follows that $\sigma=0$ and every complex algebraic $G$-manifold $X$ has a $G$-equivariant compactification by the Sumihiro theorem (see Remark \ref{RemonSumi}). So, if $G$ is a connected complex linear algebraic group, then an almost homogeneous complex algebraic $G$-manifold $X$ is $(b,0)$-compactifiable in $Al_{G, sm}^{a}$ if and only if $X$ has $b$ topological ends. 
\end{remark}

\begin{remark}\label{remar13}
Let $X$ be an almost homogeneous complex algebraic $G$-manifold with open $G$-orbit $\Omega$, and let $Y$ be a fiber of the Albanese map $\alpha\colon X\to G/I$. Since $X=G\times^{I}Y$, it follows the following statement: $Y$ has a compactification in the category $\mathcal{A}l_{I,sm}^{a}$ if and only if $X$ has a compactification in the category $\mathcal{A}l_{G,sm}^{a}$.
\end{remark}

\section{Hartogs for $(1,\sigma)$-compactified pairs}\label{sec4}
Let $X'$ be a compact complex analytic variety, $Z\subset X'$ be a proper closed analytic set, $X:=X'\setminus Z$, $i\colon X \hookrightarrow X'$ be the open immersion, and $j\colon Z\hookrightarrow X'$ be the closed immersion. 

Consider the corresponding functors: 
\begin{gather*}
i_{!}\colon Sh_{\mathbb{C}}(X)\to Sh_{\mathbb{C}}(X')\\ 
i^{-1}\colon Sh_{\mathbb{C}}(X')\to Sh_{\mathbb{C}}(X)\\ 
j_{*}\colon Sh_{\mathbb{C}}(Z)\to Sh_{\mathbb{C}}(X')\\  
j^{-1}\colon Sh_{\mathbb{C}}(X')\to Sh_{\mathbb{C}}(Z).
\end{gather*}

We have the following natural morphisms of functors: 
\begin{gather*}
i_{!}\circ i^{-1}\rightarrow 1_{Sh_{\mathbb{C}}(X')}\\  1_{Sh_{\mathbb{C}}(X')}\rightarrow j_{*}\circ j^{-1}.
\end{gather*}

Also we have the following global section functors 
\begin{gather*}
\Gamma(X',-)\colon Sh_{\mathbb{C}}(X')\to Mod(\mathbb{C})\\ \Gamma_{c}(X,-)\colon Sh_{\mathbb{C}}(X)\to Mod(\mathbb{C})\\\Gamma(Z,-)\colon Sh_{\mathbb{C}}(Z)\to Mod(\mathbb{C}).
\end{gather*}

Since $\Gamma(X',-)\circ i_{!}=\Gamma_{c}(X,-), \Gamma(X',-)\circ j_{*}=\Gamma(Z, -)$, it follows the following natural morphisms of functors: 
\begin{gather*}
\Gamma_{c}(X,-)\circ i^{-1}\rightarrow \Gamma(X',-) \\
\Gamma(X',-)\rightarrow \Gamma(Z,-)\circ j^{-1}
\end{gather*}

Note that, for any injective object $I\in Sh_{\mathbb{C}}(X')$ we have the following exact sequence 
\begin{gather*}
0\to \Gamma_{c}(X,i^{-1}I)\to \Gamma(X',I)\to \Gamma(Z, j^{-1}I)\to 0.
\end{gather*}

This implies the following distinguished triangle for any sheaf $\mathcal{F}'\in Sh_{\mathbb{C}}(X')$: 

\begin{equation}
\mathbf{R}\Gamma_{c}(X,i^{-1}\mathcal{F}')\to \mathbf{R}\Gamma(X',\mathcal{F}')\to \mathbf{R}\Gamma(Z,j^{-1}\mathcal{F}')\to_{+1}
\end{equation}

We have the following commutative diagrams of morphisms of functors on $Sh(X')$:

\[
\begin{diagram}
\node{\Gamma_{c}(X,-)\circ i^{-1}} \arrow{e,t}{} \arrow{se,r}{} \node{\Gamma(X',-)} \arrow{s,r}{}
\\
\node[2]{\Gamma(X,-)\circ i^{-1}}\\
\end{diagram}\quad 
\begin{diagram}
\node{\Gamma(X',-)}\arrow{s,r}{} \arrow{e,t}{} \node{\Gamma(Z,-)\circ j^{-1}} \arrow{s,r}{}
\\
\node{\Gamma(X,-)\circ i^{-1}} \arrow{e,t}{} \node{\Gamma(\partial X,-)\circ i^{-1}} \\
\end{diagram}
\]
This implies the following commutative diagram of distinguished triangles for any sheaf $\mathcal{F}'\in Sh_{\mathbb{C}}(X')$: 
\[
\begin{diagram} 
\node{\mathbf{R}\Gamma_{c}(X,i^{-1}\mathcal{F}')} \arrow{e,t}{} \arrow{s,r}{=}
\node{\mathbf{R}\Gamma(X',\mathcal{F}')} \arrow{e,t}{} \arrow{s,r}{}
\node{\mathbf{R}\Gamma(Z,j^{-1}\mathcal{F}')\to_{+1}} \arrow{s,r}{}
\\
\node{\mathbf{R}\Gamma_{c}(X, i^{-1}\mathcal{F}')} \arrow{e,t}{}
\node{\mathbf{R}\Gamma(X,i^{-1}\mathcal{F}')} \arrow{e,t}{}
\node{\mathbf{R}\Gamma(\partial X, i^{-1}\mathcal{F}')\to_{+1}}
\end{diagram}
\]

Further, to the end of Section \ref{sec4}, we consider only the following categories $B_T$ as in Definition \ref{remoncategory}: $B$ is one of the following $An_{norm}$, $An_{sm}$, $Al_{norm}$, $Al_{sm}$, $An_{G,norm}^{a}$, $An_{G,sm}^{a}$, $Al_{G,norm}^{a}$, $Al_{G,sm}^{a}$; $T$ is one of the following $Id$ or $Sh_{\mathbb{C}}$.  In this case, $X$ has exactly $b$ topological ends. Suppose $(X,\mathcal{F})$ is $(b,\sigma)$-compactified by the morphism  $(X,\mathcal{F})\to(X',\mathcal{F}')$ in $B_T$.

It follows the following commutative diagram: 

\[
\begin{diagram} 
\node{\mathbf{R}\Gamma_{c}(X,\mathcal{F})} \arrow{e,t}{} \arrow{s,r}{\cong}
\node{\mathbf{R}\Gamma(X',\mathcal{F}')} \arrow{e,t}{} \arrow{s,r}{}
\node{\mathbf{R}\Gamma(Z,j^{-1}\mathcal{F}')\to_{+1}} \arrow{s,r}{}
\\
\node{\mathbf{R}\Gamma_{c}(X, \mathcal{F})} \arrow{e,t}{}
\node{\mathbf{R}\Gamma(X,\mathcal{F})} \arrow{e,t}{}
\node{\mathbf{R}\Gamma(\partial X, \mathcal{F})\to_{+1}}
\end{diagram}
\]
and for cohomologies we obtain the following commutative diagram
\begin{equation}\label{diag5}
\begin{diagram} 
\node{\Gamma(X',\mathcal{F}')} \arrow{e,t}{} \arrow{s,r}{}
\node{\Gamma(Z,j^{-1}\mathcal{F}')} \arrow{e,t}{} \arrow{s,r}{}
\node{H^{1}_{c}(X,\mathcal{F})} \arrow{s,r}{\cong} \arrow{e,t}{} \node{H^{1}(X',\mathcal{F}')} \arrow{s,r}{}
\\
\node{\Gamma(X,\mathcal{F})} \arrow{e,t}{} 
\node{\Gamma(\partial X,\mathcal{F})} \arrow{e,t}{}
\node{H^{1}_{c}(X,\mathcal{F})} \arrow{e,t}{} \node{H^{1}(X,\mathcal{F})}
\end{diagram}
\end{equation}

\subsection{Case of $(1,\sigma)$-compactifiability of $(X,\mathcal{F})$} 

Let $X\in B$, and let $\mathcal{F}$ be a locally free $\mathcal{O}_{X}$-module of finite rank. Let us assume that $(X,\mathcal{F})$ is $(1,\sigma)$-compactified by the morphism $(X,\mathcal{F})\to (X',\mathcal{F}')$ in $B_{Sh_{\mathbb{C}}}.$ Consider the following assertions: 

\begin{enumerate}
\item $\mathcal{F}$ is Hartogs;
\item The canonical homomorphism $\Gamma(X',\mathcal{F}')\to \Gamma(Z, j^{-1}\mathcal{F}')$ is an epimorphism.;
\item $\dim_{\mathbb{C}} H^{1}_{c}(X,\mathcal{F})\leq \sigma$;
\item $\dim_{\mathbb{C}} H^{1}_{c}(X,\mathcal{F})<\infty$.
\end{enumerate}

\begin{theorem}\label{thm30} 
With the above data, we obtain the following assertions.
\begin{enumerate}
\item We have the following implications: $1 \Rightarrow 2\Rightarrow 3 \Rightarrow 4$;
\item If $\sigma=0$, then $1\Leftrightarrow 2\Leftrightarrow 3\Rightarrow 4$;
\item If $\mathcal{F}=\mathcal{O}_{X}$, then $1\Leftrightarrow 2\Leftrightarrow 3\Leftrightarrow 4$;
\item If $X$ is nonsingular, each compact irreducible divisor of $X$ is $\mathcal{F}$-removable, $\Gamma(\partial X, \mathcal{O}_{X})\neq \mathbb{C}$ and $\Gamma(\partial X, \mathcal{F})\neq 0$, then $1 \Leftrightarrow 2\Leftrightarrow 3\Leftrightarrow 4$. 
\end{enumerate}
\end{theorem} 

\begin{proof}
\begin{enumerate}
\item The implication $2 \Rightarrow 3$ follows from the first line of the diagram \ref{diag5}. Further, since the first square of the diagram $\ref{diag5}$ is Cartesian, it follows that the implication $1 \Rightarrow 2$ is also true. The implication $3\Rightarrow 4$ is clear. 
\item The implication $3 \Rightarrow 1,2$ is clear (see diagram \ref{diag5}).
\item The equivalence $1 \Leftrightarrow 4$ follows from Theorem \ref{th1'}.
\item The implication $4\Rightarrow 1$ follows from Theorem \ref{mainthsheaf}.
\end{enumerate}
\end{proof}

In particular, we obtain the following Lefschetz type property.

\begin{corollary}\label{lef}
With the above data, if $\mathcal{F}$ is Hartogs, then for any sheaf of $\mathbb{C}$-vector spaces $\mathcal{F}'$ on $X'$ such that $i^{-1}\mathcal{F}'\cong \mathcal{F}$ and $\dim H^{1}(X',\mathcal{F}') <\infty$, the canonical homomorphism $\Gamma(X',\mathcal{F}')\to \Gamma(Z,j^{-1}\mathcal{F}')$ is surjective.\end{corollary}

\begin{remark}\label{lefrem}
In particular, for any locally free $\mathcal{O}_{X'}$-module of finite rank, we obtain that if $i^{-1}\mathcal{F}'$ is Hartogs, then the canonical homomorphism $\Gamma(X',\mathcal{F}')\to \Gamma(Z,j^{-1}\mathcal{F}')$ is surjective. For the case where $X'$ is a projective complex manifold, $Z$ is a hyperplane, and $\mathcal{F}'$ is an algebraic vector bundle, we obtain that $i^{-1}\mathcal{F}'$ is Hartogs because $X$ is Stein; so, we obtain the Grauert-Grothendieck theorem on the Lefschetz property for algebraic vector bundles over projective manifolds (see \cite[Theorem 3.1]{Bost} or \cite[pg. 83]{Gro}). 
\end{remark}

For the case $\mathcal{F}=\mathcal{O}_{X}$ and $Z$ is an analytic set of codimension one of a projective complex manifold $X'$, we can consider the sheaves $\mathcal{F'}$ as the line bundles $\mathcal{O}_{X'}(-mD)$ where $D$ is an effective divisor supported on $Z$, $m\in\mathbb{Z}_{>0}$. For the nef divisors $D$, we obtain the following characterization of the Hartogs property for $X$ (see \cite{Fek3} for more details). 

\begin{corollary}\label{nefcase}
Let $X'$ be a complex projective manifold, $\dim X'>1$, $Z$ be a connected analytic subset of codimension one which is the support of a nef effective Cartier divisor $D$ on $X'$, and $X:=X'\setminus Z$. The following assertions are equivalent: 

\begin{itemize}
\item $X$ is not Hartogs;

\item $H^{1}_{c}(X,\mathcal{O}_{X})$ is $\infty$-dimensional;

\item $D$ is abundant of Iitaka dimension one;

\item $D$ is semiample of Iitaka dimension one;

\item $X$ is a proper fibration over an affine curve;
\end{itemize}
\end{corollary}

Consider $X\in B$ and let us assume that $X$ is $(1,\sigma)$-compactified by $X\to X'$ in $B$. Taking $\mathcal{F}'=\mathcal{O}_{X'}$, we obtain the following corollary. 

\begin{corollary}\label{thm31}
The following assertions are equivalent:
\begin{enumerate}
\item $\mathcal{O}_{X}$ is Hartogs;
\item $\Gamma(Z, j^{-1}\mathcal{O}_{X'})\cong \mathbb{C}$;
\item $\dim_{\mathbb{C}} H^{1}_{c}(X,\mathcal{O}_{X})\leq \sigma$;
\item $\dim_{\mathbb{C}} H^{1}_{c}(X,\mathcal{O}_{X})<\infty$.
\end{enumerate}
\end{corollary}

\begin{remark}
In the case of $B=Al_{sm}^{a}$, we have $\sigma=\dim_{\mathbb{C}} \operatorname{Alb}(X)$ (see Propositions \ref{prop27}). Moreover, if $X$ is an almost homogeneous algebraic $G$-manifold with an open $G$-orbit $\Omega$ (where $G$ is a connected complex algebraic group), then $\sigma=\dim_{\mathbb{C}} \operatorname{Alb}(\Omega)$ (see Section \ref{sectionalmost}).
\end{remark}

\subsection{Case of $(b,\sigma)$-compactifiability of $(X,\mathcal{O}_{X}), b>1$}

\begin{proposition}\label{prop33}
Suppose $X\in B$ is $(b,\sigma)$-compactified by $X\to X'$ in $B$, $b>1$. Let $\{E_{i}\}$ be the set of connected components of $Z=X'\setminus X$. If there exists $i$ such that $\Gamma(E_i, \mathcal{O}_{X'}|_{E_i})\cong\mathbb{C}$, then $\mathcal{O}_X$ is Hartogs. 
\end{proposition}

\begin{proof}
By Corollary \ref{thm31}, it follows that the structure sheaf $\mathcal{O}_{X''}$ of $X'':=X'\setminus E_{i}$ is Hartogs. By Theorem \ref{th1'}, it follows that $\mathcal{O}_X$ is Hartogs. 
\end{proof}

\textbf{Question}: Suppose $X\in B$ is $(b,\sigma)$-compactified by $X\to X'$ in $B$, $b>1$. Let $\{E_{i}\}$ be a set of connected components of $Z=X'\setminus X$. Is it true that $\mathcal{O}_X$ is Hartogs implies that there exists $i$ such that $\Gamma(E_i, \mathcal{O}_{X'}|_{E_i})\cong\mathbb{C}$?

\newpage
\section{Almost homogeneous algebraic $G$-varieties}\label{sec5}
\subsection{Calculations for $\Gamma(Z,\mathcal{O}_{X'}|_{Z})$}

Let $X\in Al_{G,sm}^{a}$ be an almost homogeneous algebraic $G$-manifold with open $G$-orbit $\Omega=G/H$ (here $G$ is a connected complex algebraic group). Suppose $X$ is $(1,\sigma)$-compactifiable in $Al_{G,sm}^{a}$. It follows that $X$ has only one topological end and $\dim_{\mathbb{C}}  \operatorname{Alb}(\Omega)=\sigma$ (see Remark \ref{remar12}). Let $X\hookrightarrow X'$ be any compactification in $Al_{G,sm}^{a}$. 

Note that $X=G\times^{I}Y$ and $X'=G\times^{I}Y'$, where $I=G_{aff}H$, $Y$ is a complex $I$-manifold (moreover, $Y$ is an almost homogeneous $G_{aff}$-manifold), and $Y'$ is an $I$-equivariant compactification of $Y$ (see Section \ref{sectionalmost}). 

Let $Z:=X'\setminus X, F:=Y'\setminus Y$, and $q\colon G\times Y'\to X'$ be the canonical quotient map. Since $\pi^{-1}(X)=G\times Y$, it follows that $\pi^{-1}(Z)=G\times F$. 

Now for any neighborhood $U$ of $Z$, we have $G\times F\subset q^{-1}(U).$ Moreover, $\bigcap\limits_{U\supset Z}q^{-1}(U)=G\times F.$ It follows that the family $\{q^{-1}(U)\}$ forms a base of $I$-invariant open neighborhoods of $G\times F$. On the other hand, the set $G\times F$ admits a base of all neighborhoods of the form $G\times V$, where $\{V\}$ is a base of all neighborhoods of $F$. 

Recall that for any open set $U\subset X'$ we have $\Gamma(U,\mathcal{O}_{X'})=(\Gamma(q^{-1}(U),\mathcal{O}_{G\times Y'}))^{I}.$ 

Now we obtain the following canonical injective homomorphisms, where $\widehat{\otimes}$ is the topological tensor product (see \cite{Dem, Schaefer}):
$$\Gamma(q^{-1}(U),\mathcal{O}_{G\times Y'})\hookrightarrow\varinjlim\limits_{V\supset F}\Gamma(G\times V,\mathcal{O}_{G\times Y'})\cong\varinjlim\limits_{V\supset F}\Gamma(G,\mathcal{O}_{G})\widehat{\otimes}\Gamma(V,\mathcal{O}_{Y'})\hookrightarrow \Gamma(G,\mathcal{O}_{G})\widehat{\otimes}\Gamma(F,\mathcal{O}_{Y'}|_{F}).$$

Taking $I$-invariants and the direct limit over neighborhoods $U$ of $Z$ for both sides, we obtain the following canonical injective homomorphism: 

$$\Gamma(Z,\mathcal{O}_{X'}|_{Z})=\varinjlim\limits_{U\supset Z}(\Gamma(q^{-1}(U),\mathcal{O}_{G\times Y'}))^{I}\hookrightarrow(\Gamma(G,\mathcal{O}_{G})\widehat{\otimes}\Gamma(F,\mathcal{O}_{Y'}|_{F}))^{I}.$$

Note that the algebra of regular functions $\mathbb{C}[G]$ is a dense subspace of the algebra of holomorphic functions $\Gamma(G,\mathcal{O}_{G})$ with respect to the canonical Fr\'echet topology. 

Now, assume that there exists an $I$-invariant open algebraic subvariety $S\subset Y'$ such that $S\supset F$ and $\mathbb{C}[S]$ is a dense subspace in $\Gamma(F,\mathcal{O}_{Y'}|_{F})$ with respect to the canonical inductive limit topology, then we have the following canonical topological isomorphisms: 
\begin{gather*}
\Gamma(G,\mathcal{O}_{G})\widehat{\otimes}\Gamma(F,\mathcal{O}_{Y'}|_{F})\cong\mathbb{C}[G]\widehat{\otimes}\mathbb{C}[S]\cong \Gamma(G\times S, \mathcal{O}_{G\times S}).
\end{gather*}

Since $G\times S=q^{-1}(q(G\times S))$, it follows the following canonical topological isomorphism: 
\begin{gather*}
\Gamma(Z,\mathcal{O}_{X'}|_{Z})\cong (\mathbb{C}[G]\widehat{\otimes}\mathbb{C}[S])^{I}.
\end{gather*} 

\subsection{Case of semiabelian variety $G$}

Let $G$ be a semiabelian complex algebraic variety. This means that $G$ is an extension of an abelian variety $A$ by an algebraic torus $T\cong(\mathbb{C}^{*})^{n}$. 

Let $X$ be an almost homogeneous algebraic $G$-manifold with open $G$-orbit $\Omega\cong G$. In this case, $X=G\times^{T}Y$, where $Y$ is a nonsingular toric variety (about toric varieties, see, for instance, \cite{Cox}). 

Each toric variety $Y$ admits a $T$-equivariant compactification $Y'$ and $\operatorname{Alb}(T)$ is a point. It follows that if $Y$ has only one topological end, then $Y$ is a $(1,0)$-compactifiable in $Al_{T,sm}^{a}$. 

The fibration lemma \cite[Section 3]{Gilliganends} implies that if $Y$ has only one topological end, then $X$ has only one topological end. 

Let $\mathfrak{X}(T)$ be a character lattice of the torus $T$ and $\mathfrak{X}(T)_{\mathbb{R}}:=\mathfrak{X}(T)\otimes_{\mathbb{Z}}\mathbb{R}$. Let $G_{ant}\subset G$ be the largest antiaffine subgroup (equivalently, the smallest normal subgroup of $G$ having an affine quotient, or kernel of affinization map $G\to\operatorname{Spec} (\mathbb{C}[G])$). 

The multiplication map $(T\times G_{ant})/(G_{ant})_{aff}\to G$ is an isogeny with a finite kernel $N=(T\cap G_{ant})/(G_{ant})_{aff}$ (see \cite[Remark 5.1.2.]{Brion3}). Here the groups $T\cap G_{ant}, (G_{ant})_{aff}$ are viewed as subgroups of $T\times G_{ant}$ via $x\to (x,x^{-1})$. It follows the following canonical isomorphism $\mathbb{C}[G]=(\mathbb{C}[T]^{(G_{ant})_{aff}})^{N}.$ 

Since the group $(G_{ant})_{aff}$ is a connected closed algebraic subgroup of the torus $T$, it follows that $(G_{ant})_{aff}$ is also a torus, which is the intersection of kernels of some characters of $T$ \cite[Section 16]{Humph}: 
\begin{gather*}
(G_{ant})_{aff}=\bigcap\limits_{i=1}^{k}\ker (l_{i}\colon T\to\mathbb{C}^{*}).
\end{gather*}

Let $L_{0}:=\mathbb{Z}\langle l_{i}\mid i=1,\cdots n\rangle $ be a sublattice of $\mathfrak{X}(T)$ spanned by $l_{i}$. Note that $\mathbb{C}[T]=\bigoplus\limits_{l\in \mathfrak{X}(T)}\mathbb{C} t^{-l},$ where $t^{-l}$ is a monomial of the weight $l\in \mathfrak{X}(T)$. It follows that $t^{-l}$ is $(G_{ant})_{aff}$-invariant if and only if $l|_{(G_{ant})_{aff}}\equiv 1$ if and only if $ l\in L_0$. This implies that 
\begin{gather*}
\mathbb{C}[G]=\bigg(\bigoplus\limits_{l\in L_0}\mathbb{C} t^{-l}\bigg)^{N}.
\end{gather*}

Note that the lattice $L_{0}$ is the character lattice of the torus $T/(G_{ant})_{aff}$ and $N$ is a finite subgroup of $T/(G_{ant})_{aff}$. Suppose the elements of $N$ are represented by $\xi_{i}\in T\cap G_{ant}, i=1,\cdots, m$. For any $l\in L_0$, the monomial $t^{-l}$ is $N$-invariant if and only if $l(\xi_{i})=1, \forall i=1,\cdots,m$. 

Consider the lattice $L:=\{l\in L_{0}\mid l(\xi_{i})=1, \forall i=1,\cdots,m\}$ (which is a finite index sublattice of $L$). It follows that 

\begin{gather*}
\mathbb{C}[G]=\bigoplus\limits_{l\in L}\mathbb{C}t^{-l}.
\end{gather*}

Assume that $Y$ has only one topological end. There exists a base $\{V\}$ of all neighborhoods of $F$ such that each $V$ is invariant with respect to the real compact form $K\cong (S^{1})^{n}$ of torus $T$ (see, for instance, \cite[Section 2.2, Lemma 2]{Akhiezer}). Moreover, there exists an open toric subvariety $S\subset Y'$ (here $Y'$ is a $T$-equivariant compactification of $Y$) such that $F\subset S$ and $\mathbb{C}[S]$ is a dense subspace of $\Gamma(V,\mathcal{O}_{Y'})$ for any $K$-invariant neighborhood $V$ (see, for instance, \cite[Section 4]{Fek1}). 

Let $\Sigma$ be the fan of the toric variety $Y$, and $|\Sigma|$ be the support of $\Sigma$. This implies that 
\begin{gather*}
\mathbb{C}[S]=\bigoplus\limits_{l\in \mathfrak{X}(T)\cap C}\mathbb{C}t^{-l},
\end{gather*} where $C=(\overline{\mathfrak{X}(T)_{\mathbb{R}}\setminus\mid\Sigma\mid})^{\vee}\subset \mathfrak{X}(T)_{\mathbb{R}}$ is a strictly convex rational cone which is the dual cone of the closed cone $\overline{\mathfrak{X}(T)_{\mathbb{R}}\setminus\mid\Sigma\mid}$ (see, for instance, \cite[Section 6]{Fek1}).

Note that $$\big(\mathbb{C}[G]\widehat{\otimes}\mathbb{C}[S]\big)^{T}=\bigg(\big(\bigoplus\limits_{l\in L}\mathbb{C}t^{-l}\big)\widehat{\otimes}\big(\bigoplus\limits_{l'\in \mathfrak{X}(T)\cap C}\mathbb{C}t'^{-l'}\big)\bigg)^{T}=\bigg(\widehat{\bigoplus\limits_{(l,l')\in L \times (\mathfrak{X}(T)\cap C)}}\mathbb{C}t^{-l}\otimes t'^{-l'}\bigg)^{T},$$ where $\widehat{\bigoplus}$ is the completed direct sum.

Algebraic torus $T$ acting in this space by the following formula: $\xi. (t^{-l}\otimes t'^{-l'})= (l+l')(\xi)(t^{-l}\otimes t'^{-l'}).$

The monomial $t^{-l}\otimes t^{-l'}$ is $T$-invariant if and only if $l+l'=0$. It follows that: 
\begin{gather*}
(\mathbb{C}[G]\widehat{\otimes}\mathbb{C}[S])^{T}=\widehat{\bigoplus\limits_{l\in L\cap (-C)}}\mathbb{C}t^{-l}\otimes t'^{l}.
\end{gather*}

We obtain the following convex-geometric criterion for the Hartogs phenomenon:

\begin{theorem}\label{semitoric}
Let $G$ be a semiabelian variety with the abelian variety $A$ and torus $T$, and let $L\subset\mathfrak{X}(T)$ be the sublattice of the character lattice of $T$ such that $\mathbb{C}[G]=\bigoplus\limits_{l\in L}\mathbb{C}t^{-l}$. Let $X$ be a noncompact almost homogeneous complex algebraic $G$-manifold with open orbit $G$. Assume that the fiber $Y$ of the Albanese map $\alpha\colon X\to A$ is a toric variety with a fan $\Sigma$ such that $\overline{\mathfrak{X}(T)_{\mathbb{R}}\setminus\mid\Sigma\mid}$ is a connected set. Then $\mathcal{O}_{X}$ is Hartogs if and only if $L\cap (\overline{\mathfrak{X}(T)_{\mathbb{R}}\setminus\mid\Sigma\mid})^{\vee}=0$.
\end{theorem}

\begin{corollary}
Let $G, A, T, X, Y$ be as in Theorem \ref{semitoric}. If $\mathcal{O}_Y$ is Hartogs, then $\mathcal{O}_{X}$ is Hartogs. 
\end{corollary}

\begin{proof}
The sheaf $\mathcal{O}_{Y}$ is Hartogs if and only if $(\overline{\mathfrak{X}(T)_{\mathbb{R}}\setminus\mid\Sigma\mid})^{\vee}=0$ (see, for instance, \cite[Section 6.1]{Fek1} or \cite[Theorem A]{Fek2}).
\end{proof}

Note that if $\mathcal{O}_Y$ is not Hartogs, then $\mathcal{O}_{X}$ may or may not be Hartogs (see example below).

\subsection{Example}
Let $\{e_1,e_2,e_3\}$ be a basis of $\mathbb{C}^{3}$, and let $\Lambda=\mathbb{Z}\langle e_{1},e_{2},e_{3},i(e_1+e_3)\rangle$ be a lattice spanned by $e_{1},e_{2},e_{3},i(e_1+e_3)$. Note that $G=\mathbb{C}^{3}/\Lambda$ is an extension of the elliptic curve $A=\mathbb{C}/\mathbb{Z}\langle1,i\rangle$ by torus $(\mathbb{C}^{*})^{2}$. Indeed, the projection $p\colon\mathbb{C}\langle e_1,e_2,e_3\rangle\to\mathbb{C}\langle e_3\rangle$ induces the surjective map 
\begin{gather*}
\alpha\colon G\to \mathbb{C}\langle e_3\rangle/\mathbb{Z}\langle e_{3},ie_{3}\rangle\cong \mathbb{C}/\mathbb{Z}\langle1,i\rangle.
\end{gather*} 
The kernel of $\alpha$ is exactly 
\begin{gather*}
T=\ker p/\ker p\cap \Lambda=\mathbb{C}\langle e_{1},e_{2}\rangle/\mathbb{Z}\langle e_{1},e_{2}\rangle \cong (\mathbb{C}^{*})^{2}.
\end{gather*}

Now let $z_{1},z_{2},z_{3}$ be holomorphic coordinates corresponding to the basis vectors $e_1,e_2,e_3$. Let $\pi\colon \mathbb{C}^{3}\to G$ be the quotient map. If $p=(z_{1},z_{2},z_{3})$, $ q=(z_{1}',z_{2}',z_{3}')$ are points of $\mathbb{C}^{3}$, then $\pi(p)=\pi(q)$ if and only if there exist $a_1,a_2,a_3,a_4\in\mathbb{Z}$ such that 
\begin{equation}\label{eq}
z_{1}-z_{1}'=a_1+ia_4, z_{2}-z_{2}'=a_2,z_{3}-z_{3}'=a_{3}+ia_4
\end{equation}

Each $f\in \mathbb{C}[G]$ can be represented by a polynomial in $e^{i z_{1}}, e^{i z_{2}}$: 
\begin{gather*}
f=\sum\limits_{(n_1,n_2)\in\mathbb{Z}^{2}}a_{n_1,n_2}e^{in_1 z_{1}}e^{in_2 z_{2}}.
\end{gather*}

By the formulas (\ref{eq}), it follows that for any $a_1,a_2,a_4\in\mathbb{Z}$ we would have the following equalities:

$$\sum\limits_{(n_1,n_2)\in\mathbb{Z}^{2}}a_{n_1,n_2}e^{in_{1}z_{1}}e^{in_{2}z_{2}}=\sum\limits_{(n_1,n_2)\in\mathbb{Z}^{2}}a_{n_1,n_2}e^{in_{1}(z_{1}+a_{1}+ia_{4})}e^{in_{2}(z_{2}+a_2)}=\sum\limits_{(n_1,n_2)\in\mathbb{Z}^{2}}a_{n_1,n_2}e^{-n_{1}a_4} e^{in_{1}z_{1}} e^{in_{2}z_{2}}$$

It follows that for a given $n_1\in\mathbb{Z}$ we would have $e^{-n_1a_4}=1$ for all $a_{4}\in\mathbb{Z}$. Hence $n_1=0$ and $L=\{(0,n_2)\in\mathbb{Z}^{2}\}$. 

Let $w_{1},w_{2}$ be holomorphic coordinates in $\mathbb{C}^{2}$. Consider the standard algebraic torus $(\mathbb{C}^{*})^{2}\subset \mathbb{C}^{2}$, and isomorphism $\operatorname{exp}\colon T\cong (\mathbb{C}^{*})^{2}$ defined by $w_{1}=e^{iz_{1}}, w_{2}=e^{iz_{2}}$ (here $z_{i}$ are holomorphic coordinates in $\mathbb{C}^{3}$ as above). So, $\mathbb{C}[G]=\mathbb{C}[z_{2}, z_{2}^{-1}]$.

Moreover, $ \operatorname{Spec}(\mathbb{C}[G])=\mathbb{C}\langle e_2\rangle/\mathbb{Z}\langle e_{2}\rangle\cong\mathbb{C}^{*}$. The affinization map $G\to  \operatorname{Spec}(\mathbb{C}[G])$ induced by projection $\mathbb{C}\langle e_{1},e_{2},e_{3}\rangle \to \mathbb{C}\langle e_{2}\rangle$. It follows that 
\begin{gather*}
G_{ant}=\mathbb{C}\langle e_{1},e_{3}\rangle /\mathbb{Z}\langle e_{1},e_{3}, i(e_{1}+e_{3})\rangle,\\  \operatorname{Alb}(G_{ant})=\mathbb{C}\langle e_{3}\rangle/\mathbb{Z}\langle e_{3},ie_{3}\rangle = \operatorname{Alb}(G),\\
(G_{ant})_{aff}=\mathbb{C}\langle e_{1}\rangle /\mathbb{Z}\langle e_{1}\rangle\cong \mathbb{C}^{*},\\ T\cap G_{ant}=(G_{ant})_{aff}.
\end{gather*}

If $Y_{1}=\mathbb{CP}^{1}_{[w_{1}:1/w_{1}]}\times \mathbb{C}_{w_{2}}^{1}$, then the sheaf of holomorphic functions of $G\times^{T}Y_{1}$ is not Hartogs. Indeed, in this case, we have 
\begin{gather*}
S=\mathbb{CP}^{1}_{[w_{1}:1/w_{1}]}\times \mathbb{C}_{1/w_{2}}^{1}
\\
\mathbb{C}[S]=\mathbb{C}[w_{2}^{-1}]
\\
C=\{(0,n_2)\in\mathbb{Z}^{2}\mid n_{2}\leq 0\}. 
\end{gather*} 
This implies that $L\cap C=C\neq 0$.

If $Y_{2}=\mathbb{C}_{w_{1}}\times \mathbb{CP}^{1}_{[w_{2}:1/w_{2}]}$, then the sheaf of holomorphic functions of $G\times^{T}Y_{2}$ is Hartogs. Indeed, in this case, we have 
\begin{gather*} 
S=\mathbb{C}_{1/w_{1}}\times \mathbb{CP}^{1}_{[w_{2}:1/w_{2}]} \\\mathbb{C}[S]=\mathbb{C}[w_{1}^{-1}]\\
C=\{(n_1, 0)\in\mathbb{Z}^{2}\mid n_{1}\leq 0\}. 
\end{gather*} 
This implies that $L\cap C=0$.

In both cases, the sheaves of holomorphic functions over $Y_{1}$ and $ Y_{2}$ are not Hartogs.

\end{document}